\documentclass[10pt,leqno,twoside]{amsart}

\usepackage{enumerate}

\usepackage{amssymb}

\usepackage[english]{babel}
\usepackage{amssymb,amsthm,amsmath,eucal,mathrsfs}
\usepackage{bm}

\usepackage{amsmath}
\usepackage{amsfonts}
\usepackage{amsmath}
\usepackage{amssymb}
\usepackage{graphicx}
\usepackage{lscape}
\usepackage{amstext}
\usepackage{amsthm}
\usepackage{color}
\usepackage{float}
\usepackage{mathrsfs}
\usepackage{epsfig}
\usepackage{url}
\usepackage{fancyhdr}
\usepackage{pspicture}
\usepackage{graphicx}

\usepackage{cite}

\usepackage{amsmath,verbatim}

\usepackage{amsthm}

\usepackage{amssymb}

\usepackage{amsfonts}

\usepackage{dsfont}

\usepackage{hyperref}

\newtheorem{theorem}{Theorem}[section]

\newtheorem{lemma}[theorem]{Lemma}

\newtheorem{proposition}[theorem]{Proposition}

\newtheorem*{Conditions}{Conditions}

\theoremstyle{definition}

\newtheorem{remark}[theorem]{Remark}

\numberwithin{equation}{section}

\newcommand{\dist}{\mathrm{dist}}      

\renewcommand{\div}{\mathrm{div}\,}    

\newcommand\restr[2]{{
  \left.\kern-\nulldelimiterspace 
  #1 
  \vphantom{\big|} 
  \right|_{#2} 
  }}

\usepackage{eucal}

\title[Mathematical imaging using bubbles as contrast agents]{Mathematical Analysis of the acoustic imaging modality using bubbles as contrast agents at nearly resonating frequencies}

\author[Dabrowski, Ghandriche and Sini]{Alexander Dabrowski$^*$, Ahcene Ghandriche$^*$ and Mourad Sini$^{\ddag}$}
\thanks{$^*$ RICAM, Austrian Academy of Sciences, Altenbergerstrasse 69, A-4040, Linz, Austria. Email: alexander.dabrowski@ricam.oeaw.ac.at. This author is supported by the Austrian Science Fund (FWF): P 30756-NBL}
\thanks{$^*$ RICAM, Austrian Academy of Sciences, Altenbergerstrasse 69, A-4040, Linz, Austria. Email: ahcene.ghandriche@ricam.oeaw.ac.at. This author is supported by the Austrian Science Fund (FWF): P 30756-NBL}
\thanks{$^{\ddag}$ RICAM, Austrian Academy of Sciences, Altenbergerstrasse 69, A-4040, Linz, Austria. Email: mourad.sini@oeaw.ac.at. This author is partially supported by the Austrian Science Fund (FWF): P 30756-NBL}

\begin{document}

\date{}

\allowdisplaybreaks

\begin{abstract} 
 We analyze mathematically the acoustic imaging modality using bubbles as contrast agents. These bubbles are modeled by mass densities and bulk moduli enjoying contrasting scales. These contrasting scales allow them to resonate at certain incident frequencies. We consider two types of such contrasts.  In the first one, the bubbles are light with small bulk modulus, as compared to the ones of the background, so that they generate the Minnaert resonance (corresponding to a local surface wave). In the second one, the bubbles have moderate mass density but still with small bulk modulus so that they generate a sequence of resonances (corresponding to local body waves).
	
	We propose to use as measurements the far-fields collected before and after injecting a bubble, set at a given location point in the target domain, generated at a band of incident frequencies and at a fixed {\it{single backscattering direction}}. Then, we scan the target domain with such bubbles and collect the corresponding far-fields. The goal is to reconstruct both the, variable, mass density and bulk modulus of the background in the target region.
\bigskip

\begin{enumerate}	
\item	We show that, for each fixed used bubble, the contrasted far-fields reach their maximum value at, incident, frequencies close to the Minneart resonance (or the body-wave resonances depending on the types of bubbles we use). Hence, we can reconstruct this resonance from our data. The explicit dependence of the Minnaert resonance in terms of the background mass density of the background allows us to recover it, i.e. the mass density, in a straightforward way. 

\item In addition, this measured contrasted far-fields allow us to recover the total field at the location points of the bubbles (i.e. the total field in the absence of the bubbles). A numerical differentiation argument, for instance, allows us to recover the bulk modulus of the targeted region as well.
\end{enumerate}
\bigskip

	The body-wave resonances are independent on the background (in contrast of the Minnaert one), hence using them we can (only) recover the total fields, as in step $(2)$ above, and then do the reconstruction using that.
\end{abstract}


\maketitle

\section{Introduction and statement of the results}

Diffusion by highly contrasted small particles is of fundamental importance in several branches of applied sciences, as for example in material sciences and imaging. In this work, we focus on the acoustic imaging modality using microscaled bubbles as contrast agents, see \cite{Abderson-al-2011, Qui-al-2009, Quaia-2007, Ilov-al-2018} for more details on related theoretical and experimental studies. We describe a modality using the contrasted scattered fields, by the targeted anomaly, measured before and after injecting microscaled bubbles. These bubbles are modeled by mass densities and bulk moduli enjoying contrasting scales. These contrasting scales allow them to resonate at certain incident frequencies. The main goal of this work is to analyze mathematically this contrasted scattered fields in terms of these scales with incident frequencies close to these resonances and derive explicit formulas linking the values of the unknown mass density and bulk modulus of the targeted region to the measured scattered fields.   

\bigskip

To describe properly the mathematical model we are dealing with in this work, let us denote by $D$ a small particle in $\mathbb{R}^{3}$ of the form $D := \varepsilon B +z $, where
 $B $ is an open, bounded, simply connected set in $\mathbb{R}^{3}$ with Lipschitz boundary, containing the origin, and $z$ specifies the location of the 
 particle. The parameter $\varepsilon > 0 $ characterize the smallness assumption on the particle. 
Let us consider a mass density (respectively, bulk modulus) that we note by $\rho_{\varepsilon}(\cdot)$ (respectively, $k_{\varepsilon}(\cdot)$) of the form 
\begin{center} 
  \begin{minipage}[b]{0.4\textwidth}
    \begin{equation*}
        \rho_{\varepsilon}(x):=\begin{cases}
        \rho_{0}(x), \ x \in \mathbb{R}^{3} \setminus D,\\ 
        \rho_{1}, \quad \; \ x \in D, \end{cases}
    \end{equation*}
  \end{minipage}
  \begin{minipage}[b]{0.5\textwidth}
    \begin{equation*}
                k_{\varepsilon}(x):=\begin{cases}
                k_{0}(x), \ x \in \mathbb{R}^{3} \setminus D,\\
                k_{1}, \quad \; \ x  \in D,
               \end{cases}
    \end{equation*} 
\end{minipage}
\end{center}
where $\rho_{1}$ and $k_{1}$ are positive constants, while $\rho_{0}$ and $k_{0}$ are smooth enough functions which are constant outside of a bounded and smooth domain $\Omega$.
We denote respectively $\bar{\rho}_0$ and $\bar{k}_0$ to be the values of $\rho_0$ and $k_0$ outside $\Omega$. 
Thus $\rho_{0}$ and $k_{0}$ denote the density and bulk modulus of the 
background medium, and $\rho_{1}$ and $k_{1}$ denote the density and bulk modulus of the bubble respectively.\\
We are interested in the following problem describing the acoustic scattering by a bubble, see \cite{Papa6} and \cite{Papa7}, given by the system
\begin{equation}\label{eq:acoustic_scattering}
\begin{cases}   \nabla \cdot \left(\dfrac{1}{\rho_{0}} \nabla u \right) +\dfrac{\omega^{2} }{k_{0}} u =0 \ \text{in} \ \mathbb{R}^{3} \setminus D,\\[10pt]
   \nabla \cdot \left(\dfrac{1}{\rho_{1}} \nabla u \right)  + \dfrac{\omega^{2} }{k_{1}} u =0 \ \text{in}\ D ,\\
   \left.u \right\vert_{-}-\left.u \right\vert_{+}=0, \ \text{on}\ \partial D,\\[4pt]
   \dfrac{1}{\rho_{1}}\left.\dfrac{\partial u}{\partial \nu} \right\vert_{-}-\dfrac{1}{\rho_{0}} \left.\dfrac{\partial u}{\partial \nu} \right\vert_{+} =0\ \text{on} \ \partial D,
   \end{cases}
\end{equation}
where $\omega > 0$ is a given frequency and $\nu $ denotes the external unit normal to $\partial D $. Here the total field is $u:=u^{i}+u^{s}$, where $u^{i}$ denotes the incident field (we restrict to plane incident waves) and $u^{s}$ denotes the scattered waves and satisfy the following condition 

\begin{equation}\label{SRC}
\frac{\partial u^{s}}{\partial |x|}-i\kappa u^{s}=o\left(\frac{1}{|x|}\right), |x|\rightarrow\infty, ~(\text{S.R.C}).
\end{equation}

We introduce the notation $\kappa_{0}^{2}:=\omega^{2} \rho_0/k_0$ and $\kappa_{1}^{2}:=\omega^{2} \rho_1 / k_1$. The problem $(\ref{eq:acoustic_scattering})$ is well posed, see \cite{AMMARI20192104, ACCS-effective-media} and \cite{az}. In addition, the scattered field $u^s $ can be expanded as 
\begin{equation*}
u^{s}(x,\theta)= \frac{e^{i \kappa_{0} \vert x \vert}}{\vert x \vert} u^{\infty}(\hat{x}, \theta)  + \mathcal{O}\left(\vert x \vert^{-2}\right), \;\; |x| \to + \infty,
\end{equation*}
where $\hat{x}:=x / |x| $ and $u^{\infty}(\hat{x}, \theta) $ denotes the far-field pattern corresponding to the unit vectors $\hat{x},\theta $, i.e. the incident and propagation directions respectively. 
We are interested in the regimes where the coefficients satisfy the conditions: 
\begin{equation*}
 \frac{\rho_1}{\rho_0}= C_{\rho} \varepsilon^s,\ s\geq 0 \;\; \; \text{and} \;\;\; \frac{k_1}{k_0}= C_{k} \varepsilon^t,\ t \geq 0,
\end{equation*}
with positive and smooth functions $C_{\rho}$ and $C_{k}$ which are independent from $\varepsilon$, and real numbers $s$, $t$ assumed to be non negative. The scattering problem described above models the acoustic wave diffracted in the presence of small bubbles. In this case, the parameters $s$ and $t$ fix the kind of medium we are considering, see \cite{Papa7, Papa6, Papa11} and \cite{az}. We are interested in the following two regimes:

\begin{enumerate}
 \item Moderate speed of propagation. In this case, we assume that $s=t$, then the relative speed of propagation is uniformly bounded, i.e. 
 \begin{equation*}
 \frac{\kappa^2_{1}}{\kappa^2_0} = \frac{\rho_{1}k_{0}}{k_{1} \rho_{0}}=\frac{\rho_{1}}{\rho_{0}}\frac{k_{0}}{k_{1}} \simeq 1, \mbox{ as } \varepsilon \ll 1.
 \end{equation*}
 
 \item High speed of propagation. In this case, we assume that $s < t $, then the relative speed of propagation is high, i.e. 
 \begin{equation*}
 \frac{\kappa^2_{1}}{\kappa^2_0} \simeq \varepsilon^{s-t}, \mbox{ as } \varepsilon \ll 1.
 \end{equation*}
\end{enumerate}

There is a major difference between these two regimes. To highlight it, let us for the moment assume that $\rho_0$ is constant everywhere in $\mathbb{R}^3$. 
In this case, the above problem can be equivalently formulated as 
\begin{equation*}
 \begin{cases}
  \Delta u + \kappa_{0}^{2} u =0 \ \text{in} \ \mathbb{R}^{3}\setminus D, \\
   \Delta u + \kappa_{1}^{2} u =0 \ \text{in}\ D,\\
   \left.u \right\vert_{-}-\left.u \right\vert_{+}=0, \ \text{on} \ \partial D ,\\[4pt]
   \dfrac{1}{\rho_{1}}\left.\dfrac{\partial u}{\partial \nu} \right\vert_{-}-\dfrac{1}{\rho_{0}} \left.\dfrac{\partial u}{\partial \nu} \right\vert_{+} =0\ \text{on} \ \partial D,\\
   u -u^i\text{ satisfies the SRC.}
 \end{cases}
\end{equation*} 

As we can see, the contrasts of the medium appear in the transmission conditions through the coefficient $1/\rho_1$ (or equivalently $\rho_0 / \rho_1$), and
through the speed of propagation, namely $\rho_0/k_0$ and $\rho_1/k_1$. Based on the Lippmann-Schwinger equation representation of the total fields, the second contrast appears 
on the (volumetric) Newtonian potential while the first one appears on the (surface) Neumann-Poincar\'e's operator. 
Precisely, the values of field $u$ outside the bubble $D$ is fully computable from the knowledge of $u(x), x\in D$ and $\partial_{\nu}u(x), x\in \partial D$. These last quantities are solutions of the following system of integral equation
\begin{equation}\label{eqNS}
  u(x) -  \gamma \omega^2  \int_D G_\omega(x-y)u(y) dy + \alpha \int_{\partial D} G_\omega (x-y) \, \partial_{\nu}u(y) \, d\sigma(y) = u^i(x), \;~~ \mbox{ on } D
\end{equation}
and
\begin{equation} \label{eqNDL}
\alpha \left(\frac{1}{\alpha} + \frac{\rho_0}{2} + ( K_D^{\omega})^* \right) \left[ \partial_{\nu} u \right] - \gamma \omega^2   \partial_{\nu-} \int_D G_\omega(x-y)u(y) dy   = \partial_{\nu} u^{i} \;~~ \mbox{ on } \partial D
\end{equation}
where $u^{i}$ is the incident field such that 
\begin{equation}\label{uincequa}
\div\left( \frac{1}{\rho_{0}} \nabla u^{i} \right)(x) + \frac{\omega^{2}}{k_{0}} u^{i}(x) = 0, \; \quad x \in \; D. 
\end{equation}
and we have adapted the succeeding notations $\gamma = \beta - \alpha \rho_1/k_1$ and $\alpha := 1/\rho_1 - 1/\rho_0$ with $ \beta := 1/k_1 - 1/k_0$.\\
Here, $G_\omega$ stands for the Green's functions related to (\ref{uincequa}) with the radiation conditions at infinity. In addition, $K_D^{\omega}$ is the double layer (or the Neumann-Poincar\'e) operator defined on the boundary of $D$ at the frequency $\omega$. Depending on the scales of the contrasts, we make the following observations.
\begin{enumerate}
 \item In the first regime, i.e $s=t$, we have $\gamma \sim 1$, as $\varepsilon <<1$, and the Newtonian potential is negligible as it scales as $\varepsilon^2$ as $\varepsilon \ll 1$. However, if $s=t=2$ then the contrasts on the mass densities, i.e. $1/\alpha$, can approximate the spectrum of the Neumann-Poincar\'e  operator $H_{D}^{0}$. For smooth domain $D$, this operator defined on $L^2(\partial D)$ has a sequence of real eigenvalues accumulating at $0$ in addition to the value $\frac{1}{2}$.  
 As the contrast is real, then we can only approximate the highest eigenvalue, which is $\frac{1}{2}$. This can be done for in this regime as $\alpha \sim \varepsilon^{-2}$. The frequency $\omega$ for which this is possible is the Minnaert resonance (corresponding to a surface wave type). 
 
 \item In the second regime, if $s<t$, the  high contrasts of the speed of propagation allow the Newtonian operator to dominate the Neumann-Poincar\'e operator.
 In addition, if we take $t-s=2$, then the contrast of the speed of propagation, $\gamma \sim \varepsilon^{-2}$, will balance the scale of the Newtonian operator and we might excite its eigenvalues. There is a discrete sequence of such eigenvalues (corresponding to local body waves type). 
\end{enumerate}

 Microbubbles with scales fitting into the first regime are well known to exist in the nature. However, those related to the second regime, with high speed of propagation, are less known. Nevertheless, there are possibilities to artificially produce them, see the discussion in \cite{Z-F-2018} and also in \cite{S-al-2019}. 

	A first key observation in our analysis, which happens to be useful for the imaging later on, is that the Minnaert resonance is characterized by the bulk modulus of the bubble and the surrounding local mass density of the background. The sequence of body-wave resonances are characterized solely by the mass density and the bulk of the bubble. In addition, we show that the contrasted scattered fields reach their maximum values at, incident, frequencies close to the Minneart resonance (or the body-wave resonances depending on the types of bubbles we use). This allows us to recover these resonances by measuring the contrasted scattered waves at a band of incident frequencies but at a fixed single backscattering direction. A second key observation is that this measured contrasted scattered waves allows us to recover the total field at the location point of the bubble. Scanning the targeted region with such bubbles, we can recover the total field there up to a sign (i.e. the total field in the absence of the bubbles). 
	\bigskip
	
	Based on these observations, we can reconstruct the density and the bulk modulus of the targeted region from
	the contrasted scattered waves (before and after injecting the bubbles) at a band of incident frequencies but at a fixed single backscattering direction. More details are given in section \ref{imaging-section}. Nevertheless, let us say it in short here that these contrasted scattered waves encodes the Minnaert resonance in its denominator and the total field in its numerator. From the first one, we extract the mass density while from the second one we derive the bulk modulus of the targeted region. 
	
\bigskip

The following theorems translate these observations with more clear statements. We state the following conditions which are common to both the two results.

\bigskip
\begin{Conditions}
Let $\Omega$ be a bounded domain of diameter $diam(\Omega)$ of order $1$. Let also $\rho_0$ and $k_0$ be two functions of class $C^1$ and are constant outside $\Omega$. They are assumed to be positive functions. Let $D:=z +\varepsilon \, B$ be a small and Lipschitz smooth domain where $z \in \Omega$ away from its boundary. The relative diameter of $D$ is small as compared to the diameter of $\Omega$, i.e. $\dfrac{\varepsilon}{diam(\Omega)}<<1$.
The functions $\rho_0$ and $k_0$ are assumed to be independent on the parameter $\varepsilon$.
\end{Conditions}

\begin{theorem}\label{Theorem-Using-Minnear-Resonance}
Let the above {\bf{Conditions}} be satisfied. In addition, let $\rho_1$ and $k_1$ be constants enjoying the following scales  
\begin{equation*}
\rho_1=\overline{\rho}_{1}\;\varepsilon^{2},\quad k_1 =\overline{k}_{1}\, \varepsilon^{2} \quad \mbox{ and } \quad \frac{k_1}{\rho_1}\sim 1, \,~~ \text{as} \;\; \varepsilon <<1
\end{equation*}
and $\overline{\rho}_{1}$ is large enough such that $\underset{x \in \Omega}{\max} \, \rho_{0}(x) < \overline{\rho}_{1}$.

The solution of the corresponding problem (\ref{eq:acoustic_scattering}), has the following expansions. 
\begin{enumerate}
\item The scattered field is approximated as
\begin{equation}\label{scat-first-regime}
 u^{s}(\cdot,\theta ,\omega) = v^{s}(x,\theta,\omega)- 
\frac{\omega^{2} \, \omega^{2}_{M}}{\overline{k}_{1}(\omega^{2}-\omega^{2}_{M})} \vert B \vert \; \varepsilon \; G_{\omega}(x-z) \, v(z,\theta,\omega) + \mathcal{O}\left( \frac{\varepsilon^{2}}{\left( \omega^2-\omega^2_{M} \right)^{2}}\right)
\end{equation}
uniformly for $x$ in a bounded domain away from $D$ and $\theta$ in the unit sphere.
\item The farfield is approximated as
\begin{equation}\label{farfield-first-regime}
 u^{\infty}(\hat{x},  \theta, \omega) =v^{\infty}(\hat{x}, \theta, \omega)- \frac{\omega^{2}_{M}}{ \overline{k}_{1} (\omega^{2} - \omega^{2}_{M})} \vert B \vert \; \varepsilon \; v(z, -\hat{x}, \omega)\; v(z, \theta, \omega) + \mathcal{O}\left( \frac{\varepsilon^{2}}{\left( \omega^2-\omega^2_{M} \right)^{2}}\right),
\end{equation}
uniformly for $\theta$ and $\hat{x}$ in the unit sphere.
\end{enumerate}
 These expansions are valid under the condition that $\varepsilon / (\omega^2 - \omega_M^2)$ small enough. Here, we have 
\begin{equation}\label{Minnaer-3D}
\omega_M:=\omega_M(z):=\sqrt{\frac{8 \pi\, \overline{k}_{1} }{\rho_0(z) \; \mu_{\partial B}}},  \quad \text{with} \quad \mu_{\partial B} := \frac{1}{\vert \partial B \vert} \;  \int_{\partial B} \; \int_{\partial B}  \frac{(x-y) \cdot \nu(x)}{ |x-y| } \; dx \; dy,
\end{equation}
called the Minnaert frequency\footnote{Remark that $\mu_{\partial B}$ depends only on the shape of the domain $B$.}. In both the expansions $v:=v(x, \theta, \omega)$ and $v^\infty:=v^\infty(\hat{x}, \theta, \omega)$ is the total field, and its farfield, solution of the problem (\ref{eq:acoustic_scattering}) in the absence of the bubble $D$.
\end{theorem}

The first mathematical study of the Minnaert resonance was shown in \cite{Habib-Minnaert} where it was estimated for bubbles injected in a homogeneous background. Later on, a series of works were devoted to its implications in different areas, see \cite{az, AMMARI20192104, ACCS-effective-media, A-F-L-Y-Z, H-F-G-L-Z-1}. The approximations in (\ref{scat-first-regime}) and (\ref{farfield-first-regime}) are extensions of those in \cite{Habib-Minnaert} and \cite{AMMARI20192104} to the case when the background is heterogeneous (with variable mass density and bulk modulus). The surprising fact is that this resonance depends also on the surrounding background through its mass density. 
\bigskip

To state the results related to the second regime, we first introduce with more details the Newtonian operator $N^0: L^{2}(B)\rightarrow L^{2}(B)$ such that $N^{0}(u)(x):=\int_{B}G_0(x-y)u(y)dy$. This operator is positive, compact and selfadjoint. Let $(\lambda^B_n, e^B_n)_{n \in \mathbb{N}}$ be its sequence of eigenvalues with the corresponding eigenfunctions.


\begin{theorem}\label{Theorem-Using-Newtonian-Resonances}
Let the above {\bf{Conditions}} be satisfied. In addition, let $\rho_1$ and $k_1$ be constants enjoying the following scales  
\begin{equation}\label{condition-rho_0}
\rho_1=\rho_0(z)+ \mathcal{O}( \varepsilon^{j}),\quad  j>0, \quad \text{and} \quad  k_1 =\overline{k}_{1}\, \varepsilon^{2} \,~~ \mbox{ as } \varepsilon <<1.
\end{equation}
In this regime, the solution of the problem (\ref{eq:acoustic_scattering}), has the following expansions.
\begin{enumerate}
\item The scattered field has the approximation 
\begin{equation}\label{Approximation-second-regime-scattered-field}
u^{s}(x, \theta, \omega) = v^s(x, \theta, \omega)-\frac{1}{\overline{k}_{1}} \; \frac{\omega^{2} \omega_{n_0}^{2}}{(\omega^2 -  \omega_{n_0}^{2})} \displaystyle\Big( \int_{B} e^B_{n_{0}} \Big)^{2}\, \varepsilon\, G_{\omega}(x;z) \; v(z, \theta, \omega) + \mathcal{O}\left(\varepsilon\; +\;\frac{\varepsilon^{1+\min(1;j)}}{\left( \omega^{2} - \omega_{n_{0}}^{2} \right)^{2}} \right),
\end{equation}
uniformly for $x$ in a bounded domain away from $D$ and $\theta$ in the unit sphere.
\item The farfield has the approximation
\begin{equation}\label{Approximation-second-regime-farfield}
u^{\infty}(x, \theta, \omega) = v^\infty(x, \theta, \omega)-\frac{1}{\overline{k}_{1}} \; \frac{\omega^{2} \omega_{n_0}^{2}}{(\omega^2 -  \omega_{n_0}^{2})} \displaystyle\Big( \int_{B} e^B_{n_{0}} \Big)^{2}\, \varepsilon\, v(z, -\hat{x}, \omega) \; v(z, \theta, \omega) + \mathcal{O}\left(\varepsilon\; +\; \frac{\varepsilon^{1+\min(1;j)}}{\left( \omega^{2} - \omega_{n_{0}}^{2}\right)^{2}} \right),
\end{equation}
uniformly for $\theta$ and $\hat{x}$ in the unit sphere.
\end{enumerate}
These expansions are valid as soon as $\dfrac{\varepsilon^h}{\omega^2-\omega^2_{n_0}}=\mathcal{O}\left( 1 \right)$, with $h<\min\{1, j\}$, as $\varepsilon <<1$, where 
\begin{equation}\label{resonance-second-regime}
\omega_{n_0}:=\sqrt{\frac{\overline{k}_{1}}{\overline{\rho}_1 \lambda^B_{n_0}}}.
\end{equation}
Observe that $\Big( \int_{B} e^B_{n_{0}} \Big)^{2}$ means $ \underset{l}{\sum}\Big( \int_{B} e^B_{l} \Big)^{2}$, where $l$ such that $N^{0}e^B_{l}=\lambda^B_{n_0}\, e^B_{l}$.

Here again $v:=v(x,\theta,\omega)$ and $v^{\infty}:=v^{\infty}(\hat{x},\theta,\omega)$ is the total field, and its farfield, solution of the problem $(\ref{eq:acoustic_scattering})$ in the absence of the bubble $D$.  
\end{theorem}
\bigskip

The body-wave resonances have been characterized already in \cite{A.D-F-M-S, M-M-S} in the framework of dielectric nanoparticles, in the scalar model related to the TM regime of the electromagnetic scattering, with a homogeneous background.
There, the contrast comes from the dielectric nanoparticles with high permittivity and moderate permeability. In our context, the contrast comes from the fact that the density of the bubble is moderate while its bulk is still small. At the mathematical level, our formulas in \ref{Approximation-second-regime-scattered-field} extend those in \cite{A.D-F-M-S} to the case of the acoustic model, i.e. a divergence form model, with heterogeneous background. As we have said above, such bubble's contrasts might not be available in nature but can be artificially designed, see \cite{Z-F-2018}.    
\bigskip

We finish this section with the following observations.
\begin{enumerate}

\item The Minnaert resonance $\omega_M$ of the bubble, located at $z$, depends on the bubble itself through its scaled bulk modulus $\overline{\rho}_1$, but most importantly on the surrounding background through its mass density $\rho_0(z)$. This is not the case with the body resonances $\omega_n, n\in \mathbb{N}$, which are fully characterized by the bubble itself through its scaled mass density and bulk modulus, compare (\ref{Minnaer-3D}) and (\ref{resonance-second-regime}). But this might be less surprising keeping in mind that the Minnaert resonance is related to the surface double layer operator, which is contrasted at the higher order of the divergence form partial differential equation, while the body-wave resonances are related to the volumetric Newtonian potential operator which contrasts at the lower order.   

\item The approximations in Theorem \ref{Theorem-Using-Minnear-Resonance} are similar to the ones in Theorem \ref{Theorem-Using-Newtonian-Resonances} up to the multiplicative factor appearing in the dominating term. The additional term $\mathcal{O}\left(\varepsilon \right)$ appearing in the error of the approximations (\ref{Approximation-second-regime-scattered-field}) and (\ref{Approximation-second-regime-farfield}) can be removed as follows:
\bigskip
 
\begin{enumerate}
\item The scattered fields are approximated as
\begin{equation*}
u^s(x, \theta, \omega)= v^s(x, \theta, \omega)+\frac{\omega^{2}}{k_1} \; G_\omega(x-z) \; v(z, \theta, \omega) \; \int_{D} W(x) dx + \mathcal{O}\left(\frac{\varepsilon^{1+\min(1;j)}}{\left( \omega^{2} - \omega_{n_{0}}^{2}\right)^{2}} \right)
\end{equation*}

\item The farfields are approximated as
\begin{equation*}
u^{\infty}(\hat{x}, \theta, \omega)=v^\infty(\hat{x}, \theta, \omega)+ \frac{\omega^{2}}{k_1} \; v(z, -\hat{x}, \omega) \; v(z, \theta, \omega) \; \int_{D} W(x) dx + \mathcal{O}\left(\frac{\varepsilon^{1+\min(1;j)}}{\left( \omega^{2} - \omega_{n_{0}}^{2}\right)^{2}} \right)
\end{equation*}

\end{enumerate}
where
$ W := \left(I - \gamma \, \omega^{2}  \, N^{0} \right)^{-1}(1)$.
The $\mathcal{O}\left( \varepsilon \right)$ appearing in (\ref{Approximation-second-regime-scattered-field}) and and (\ref{Approximation-second-regime-farfield}), is due to the fact that, see (\ref{Adz}),
\begin{equation*}
\int_{D} W(x) \, dx = -\omega^2_{n_0}\frac{\left( \int_{D} e^{D}_{n_{0}}(x) \, dx \right)^{2}}{\left( \omega^2 - \omega^2_{n_0} \right)} + \mathcal{O}\left( \varepsilon^3 \right). 
\end{equation*}

\bigskip

\item Finally, we do believe that the condition $\underset{x \in \Omega}{\max} \, \rho_{0}(x) < \overline{\rho}_{1}$ used in Theorem \ref{Theorem-Using-Minnear-Resonance} and the condition (\ref{condition-rho_0}) appearing in Theorem \ref{Theorem-Using-Newtonian-Resonances} might be removed.

\end{enumerate}

\section{An application to the acoustic imaging using resonating bubbles}\label{imaging-section}
Based on the expansions given in Theorem \ref{Theorem-Using-Minnear-Resonance} and Theorem \ref{Theorem-Using-Newtonian-Resonances}, in particular (\ref{farfield-first-regime}) and (\ref{Approximation-second-regime-farfield}), we design the following imaging procedure to reconstruct the mass density $\rho_0$ and bulk modulus $k_0$ inside the bounded domain $\Omega$ where they are variable. This procedure is based on the following measured data. Let $[\omega_{min}, \omega_{max}]$ be interval of a possible incident frequencies. We have the following conditions on this interval 
\begin{equation*}
\omega_{min} \leq \sqrt{\frac{4 \pi\, \overline{k}_{1}}{\underset{z\in \Omega}{\max}\rho_0(z) \; \mu_{\partial B} }}\leq  \sqrt{\frac{4 \pi\, \overline{k}_{1}}{\underset{z\in \Omega}{\min} \rho_0(z) \; \mu_{\partial B}}} \leq \omega_{max}.
\end{equation*}
This condition makes sense as soon as we know a priori a lower bound and an upper bound of the unknown mass density $\rho_0$.
\bigskip

\begin{enumerate}
\item Collect the farfields before injecting the bubble $D$, i.e. measure the backscattered farfield at a single incident wave $\theta$ and a band of frequencies $\omega \in [\omega_{min}, \omega_{max}]: \quad v^\infty(-\theta, \theta, \omega).$
\bigskip

\item Collect the farfield after injecting the bubble $D$, centered at the point $z \in \Omega$, i.e. measure the backscattered farfield at a single incident wave $\theta$ and a band of frequencies $\omega \in [\omega_{min}, \omega_{max}]:\quad  u^\infty(-\theta, \theta, \omega, z).$
\end{enumerate} 
\bigskip

The imaging procedure goes as follows. We set 
\begin{equation}\label{Imaging-functional}
I(\omega, z):=u^\infty(-\theta, \theta, \omega, z)-v^\infty(-\theta, \theta, \omega)
\end{equation}
 as the imaging functional, remembering that the incident angle $\theta$ is fixed. We have the following properties from (\ref{farfield-first-regime})
\begin{equation}\label{imaging-functional}
I(\omega, z) \sim -\frac{\omega^2_M}{ \overline{k}_1 (\omega^2-\omega^2_M(z))}\vert B \vert\; \varepsilon\;\; [v(z, \theta, \omega)]^2.
\end{equation}
We divide this procedure into two steps:
\begin{enumerate} 
\item Step 1. From this expansion, we recover $\omega^2_M(z)$ as the frequency for which the imaging function $\omega \rightarrow I(\omega, z)$ gets its largest value. From the estimation of this resonance $\omega^2_M(z)$, we reconstruct the mass density at the center of the injected bubble $z$, based on (\ref{Minnaer-3D}), as follows:
\begin{equation*}
\rho_0(z)= \frac{4 \pi\, \overline{k}_{1}}{ \omega^2_M(z) \; \mu_{\partial B}}. 
\end{equation*}
Scanning the domain $\Omega$ by such bubbles, we can estimate the mass density there. 
\bigskip
\item Step 2. To estimate now the bulk modulus, we go back to (\ref{imaging-functional}) or (\ref{farfield-first-regime}), and derive the values of the totale field $[v(z, \theta, \omega)]^2$. This field corresponds to the model without the bubble. 
Hence, we have at hand $v(z, \theta, \omega)$ for $z \in \Omega$ up to a sign (i.e. we know the modulus and the phase up to multiple of $\pi$). 
\bigskip

Use the equation $\nabla\cdot \rho^{-1}_0 \nabla v +\omega^2 k^{-1}_0 v=0$ to recover the values of $k_0$ in the regions where $v$ does not change sign. This can be done by numerical differentiation for instance. Other ways are of course possible to achieve this second step. In addition, we have at hand multiple frequency internal data.

\end{enumerate}
\bigskip

The procedure described above uses the Minnaert resonance. The key point to recover the mass density is the explicit dependance of this resonance on the value of the mass density on it's 'center', see (\ref{Minnaer-3D}). This is not the case for the sequence of resonances coming from the second regime, see (\ref{resonance-second-regime}). Nevertheless, using such resonances allows as to recover the internal values of total field $v(z, \theta, \omega)$, from (\ref{Approximation-second-regime-farfield}), solution of the equation $\nabla\cdot \rho^{-1}_0 \nabla v +\omega^2 k^{-1}_0 v=0$, for multiple frequencies $\omega$,
 as in Step 2. Therefore, we may recover $\rho_0$ via low frequencies and then $k_0$ via moderate frequencies, for instance. However, for technical reasons, we need to know the mass density as we use the condition (\ref{condition-rho_0}). But as we said earlier, we believe that this condition might be removed.
\bigskip

\section{Proof of Theorem \ref{Theorem-Using-Minnear-Resonance}}

We divide the proof into two steps. In the first step, we provide the expansions in the case when the background is homogeneous. This allows to show the key parts in localizing the resonance and computing the scattered fields from incident frequencies close to these resonances. In the second step, we deal with case when the background is heterogeneous  and show how this perturbation influences the derivation of the expansions and the resonances as well.    
\bigskip

Let us recall the Green's function $G_\omega$ satisfying, in the distributional sense, the equation
\begin{equation}\label{eq:defFundSol} 
\underset{x}{\nabla} \cdot \left(\frac{1}{\rho_0} \underset{x}{\nabla} G_\omega(x-z)\right) + \frac{\omega^{2}}{k_0} G_\omega(x-z) = -\delta_z(x) \quad \text{ for any } x,z \in \mathbb{R}^{3}.
\end{equation}
with the radiation conditions at infinity.

\subsection{Constant coefficients}
\label{sect:Minnaert:ConstCoeff}

We assume here that both $\rho_0$ and $k_0$ are constants everywhere in $\mathbb{R}^3$. We recall that $\rho_1 = \bar \rho_1 \varepsilon^2, \; \; k_1 = \bar k_1 \varepsilon^2$, where $\bar \rho_1, \bar k_1$ do not depend on $\varepsilon$. In this case, it is immediate to show that 
\begin{equation*}
G_\omega(x) = \rho_0\; \frac{e^{i \kappa_0 |x|}}{4 \pi |x|},
\end{equation*}
where $\kappa_{0} = \omega \sqrt{\rho_0/k_0}.$ \\
Let $u$ be the solution of \ref{eq:acoustic_scattering}. From the Lippman-Schwinger representation we have
\begin{equation}\label{eq:1}
 u( x) - \alpha \; \underset{x}{\div} \int_D G_\omega (x-y) \nabla u(y) dy -\beta \, \omega^{2} \,  \int_D G_\omega (x-y)  u(y) dy = u^i(x),
\end{equation}
where \; $\alpha := 1/\rho_1 - 1/\rho_0$ and $ \beta := 1/k_1 - 1/k_0.$\\
Since $\underset{x}{\nabla} G_\omega(x-y) = - \underset{y}{\nabla} G_\omega(x-y) $, by integration by parts and $(\ref{eq:defFundSol})$ we have
\begin{equation*}
 \underset{x}{\div} \int_D G_\omega (x-y) \nabla u(y) dy = - \frac{\omega^2 \rho_1}{k_1}   \int_D G_\omega (x-y) u(y) dy  -\int_{\partial D}  G_\omega (x-y) \partial_{\nu}u(y) d\sigma(y),
\end{equation*}
so \eqref{eq:1} becomes 
\begin{equation}\label{eq:RLS}
  u(x) -  \gamma \omega^2  \int_D G_\omega(x-y)u(y) dy + \alpha \int_{\partial D} G_\omega (x-y) \partial_{\nu} u(y) \, d\sigma(y) = u^i(x) ,
\end{equation}
where $\gamma = \beta - \alpha \rho_1/k_1.$
Taking the normal derivative as $x \to \partial D$ from inside $D$, from the jump relations of the derivative of the single layer potential we obtain
\begin{equation}\label{eq:ReLS}
\left(1 + \frac{\alpha \rho_0}{2}\right) \partial_{\nu} \, u(x) - \gamma \omega^2   \partial_{\nu-} \int_D G_\omega(x-y)u(y) dy + \alpha (K_D^{\omega})^* \left[ \partial_{\nu} \, u \right](x) =  \partial_{\nu} \, u^{i}(x),
\end{equation}
where $\left(K_D^{\omega} \right)^{\star}$ is defined by
\begin{equation*}
\left(K_D^{\omega} \right)^{\star}(f) (x) := \emph{p.v.} \;  \int_{\partial D} \frac{\partial G_{\omega}(x-y)}{\partial \nu (x)} \, f(y) \, d \sigma(y), \; f \in L^2(\partial D).
\end{equation*}
Notice that due to the scaling of $\rho_1$ and $k_1$, we have $\gamma =\mathcal{O}(1) $ as $\varepsilon \to 0$.
Expanding in $z$ the fundamental solution, we obtain for $x$ away from $D$, 
\begin{equation*}
\int_D G_\omega(x-y) u(y) dy =  G_\omega(x-z) \int_D u(y) \, dy + \mathcal{O}\left(\varepsilon^{\frac{5}{2}} \Vert u \Vert_{L^2(D)}\right),
\end{equation*}
as by the Cauchy-Schwartz inequality and the fact that $\vert y-z \vert = \mathcal{O}\left(\varepsilon \right)$ we have 
\begin{equation*}
\Big\vert \int_D (y-z) u(y) dy \Big\vert \leq \Vert \cdot - z \Vert \, \Vert u \Vert_{L^2(D)}  =    \mathcal{O}\left(\varepsilon^{\frac{5}{2}} \Vert u \Vert_{L^2(D)}\right).
\end{equation*}
In the same way, we have
\begin{equation*}
\int_{\partial D} |y-z| \;\, \partial_{\nu}u(y) \,\; d \sigma(y) \lesssim  \varepsilon^{2} \;\; \left\Vert \partial_{\nu} u \right\Vert _{L^2(\partial D)}, 
\end{equation*}
so that
\begin{equation*}
\int_{\partial D}  G_\omega(x-y) \partial_{\nu} u(y) d\sigma(y) = G_\omega(x-z) \int_{\partial D} \partial_{\nu} u(y) \, d\sigma(y) +  \mathcal{O}\left(\varepsilon^{2} \left\Vert \partial_{\nu} u \right\Vert_{L^2(\partial D)}\right).
\end{equation*}
Therefore, we can rewrite \eqref{eq:RLS} as
\begin{equation*}
u^s(x) =  \gamma \, \omega^2 \,  G_\omega(x-z) \, \int_D u(y) \, dy \, - \, \alpha \, G_\omega(x-z) \,  \int_{\partial D} \partial_{\nu} u(y) \, d\sigma(y) \, + \,   \mathcal{O}\left(\varepsilon^{\frac{5}{2}} \Vert u \Vert_{L^2(D)} +\alpha  \varepsilon^{2}\left\Vert \partial_{\nu} u \right\Vert _{L^2(\partial D)}\right).
\end{equation*}
From the equation satisfied by $u$, see for instance $(\ref{eq:acoustic_scattering})$, and the divergence theorem, we have
\begin{equation}\label{intDint} 
\int_D u(y) \, dy = - \frac{k_1}{\omega^2} \int_D \nabla \cdot \left( \frac{1}{\rho_1} \nabla u \right)(y) \, dy = - \frac{k_1}{\omega^2 \rho_1} \int_{\partial D} \partial_{\nu} u(y) \, d\sigma(y),
\end{equation}
then
\begin{equation}\label{eq:scatteredWaveAsymp}
u^s(x) =  - \left(\alpha + \frac{\gamma k_1}{\rho_1}\right) G_\omega(x-z) \int_{\partial D} \partial_{\nu} u(y) \, d\sigma(y) +   \mathcal{O}\left(\varepsilon^{\frac{5}{2}} \Vert u \Vert_{L^2(D)} + \alpha \varepsilon^{2} \left\Vert \partial_{\nu} u \right\Vert_{L^2(\partial D)}\right).
\end{equation}

\bigskip

Now, we derive the dominating term of $\int_{\partial D} \partial_\nu u \, d\sigma $ and estimate $\Vert u \Vert_{L^2(D)}$ and $\Vert \partial_\nu u \Vert_{L^2(\partial D)}$ in terms of $\varepsilon$.
Let us consider first the case when $\gamma = 0$. In this case, the equation \eqref{eq:ReLS} becomes
\begin{equation}\label{RkLippSch3}
\left( \left( 1/\alpha + \rho_0/2 \right) I + (K_D^\omega)^*\right) [\partial_\nu u ]= \alpha^{-1} \, \partial_\nu u^i,
\end{equation}
and we can rewrite it as
\begin{equation}\label{eq:RekLippmaSch}
\left( \left( 1/\alpha + \rho_0/2 \right) I + (K_D^0)^*\right)[ \partial_\nu u ]+\left( (K_D^\omega)^* - (K_D^0)^*\right)[ \partial_\nu u ] = \alpha^{-1} \, \partial_\nu u^i.
\end{equation}
Let
\begin{equation}\label{defA}
A_{\partial D} :=  \frac{\rho^2_0}{k_0 \, 4 \pi} \mu_{\partial D},\quad 
\mu_{\partial D}:=\frac{1}{\vert \partial D\vert} \int_{\partial D}\int_{\partial D} \frac{(x-y) \cdot \nu(x)}{4 \pi |x-y|} d \sigma(x) d\sigma(y)\; and \quad A(y) :=\frac{\rho^2_0}{k_0} \int_{\partial D} \frac{(x-y) \cdot \nu(x)}{4 \pi |x-y|} d \sigma(x).
\end{equation}
By the divergence theorem we have $A_{\partial D} > 0$, and it is immediate that $A - A_{\partial D}$ has average zero along $\partial D.$
Expanding $G_\omega(x-y)$ in terms of $|x-y|$, we obtain
\begin{eqnarray}\label{KwK0}
\nonumber
(K_D^\omega)^*[\partial_\nu u](x) &:=& \int_{\partial D} \frac{\partial \; G_{\omega}(x-y)}{\partial \nu(x)} \partial_\nu u( y) d\sigma(y) \\ \nonumber
 &=& \int_{\partial D} \frac{\rho_0}{4 \pi} \, \left[ \frac{(x-y)\cdot \nu (x)}{|x-y|^3} \; \big(-1 + i \kappa_0 |x-y|\big) \;  \sum_{n=0}^\infty \frac{(i \kappa_0|x-y|)^n}{n!} \right]  \partial_\nu u(y) \; d\sigma(y) \\ \nonumber
 &=& (K_D^0)^*[\partial_\nu u](x) - \frac{\kappa_0^2 \rho_0}{8 \pi} \int_{\partial D} \frac{(x-y)\cdot \nu(x)}{|x-y|} \, \partial_\nu u(y) \, d\sigma(y) \\  
&-& \frac{i \kappa_0^3 \rho_0}{12 \pi} \int_{\partial D}  (x-y) \cdot \nu(x) \, \partial_\nu u(y) \; d\sigma(y) + \mathcal{O}\left(\varepsilon^3 \Vert \partial_\nu u \Vert_{L^2(\partial D)}\right),
\end{eqnarray}
and integrating \eqref{eq:RekLippmaSch} on $\partial D$, as $K_D^0[1] = -\rho_0 / 2$, we obtain
\begin{eqnarray}\label{eq:LippSchwInt}
\nonumber
\left(\frac{1}{\alpha} - \frac{\kappa_0^2 \, k_{0}}{2 \, \rho_{0}} A_{\partial D}\right) \int_{\partial D} \partial_{\nu} u (x)\, d \sigma(x) &=& \frac{1}{\alpha} \int_{\partial D} \partial_\nu u^i(x) \, dx + \frac{i \kappa_0^3 \rho_0}{12 \pi} \int_{\partial D} \int_{\partial D} (x-y) \cdot \nu(x) \partial_\nu u(y) dy dx \\ &+& \frac{\kappa_0^3}{2 \, \rho_{0}} \int_{\partial D} \left(A(y) - A_{\partial D}\right) \partial_\nu u(y) dy   +   \mathcal{O}\left(\varepsilon^5 \Vert \partial_\nu u \Vert_{L^{2}(\partial D)}\right).
\end{eqnarray}
We can estimate the integral which contains $A(\cdot) - A_{\partial D}$ by rewriting
\begin{eqnarray}\label{ll}
\nonumber
\int_{\partial D} (A(y) - A_{\partial D}) \partial_\nu u(y) d\sigma(y) & \stackrel{\ref{RkLippSch3}}{=} & \alpha^{-1} \, \int_{\partial D}(A(y) - A_{\partial D}) \left( (\rho_0/2 + 1/\alpha + (K_D^\omega)^*)^{-1}[\partial_\nu u^i] \right) (y) \, d\sigma(y) \\ \nonumber
&=& \alpha^{-1} \, \int_{\partial D} \left((\rho_0/2 + 1/\alpha + K_D^\omega)^{-1}[A(\cdot) - A_{\partial D}] \right) (y) \; \partial_\nu u^i(y) \; d\sigma(y)  \\
& \leq & \alpha^{-1} \, \left\Vert (\rho_0/2 + 1/\alpha + K_D^\omega)^{-1}[A(\cdot) - A_{\partial D}]\left\Vert_{L^2(\partial D)} \right\Vert \partial_\nu u^i \right\Vert_{L^2(\partial D)}
= \mathcal{O}\left( \varepsilon^6 \right),
\end{eqnarray}
the last equality being a consequence of the fact that
$(\rho_0/2 + 1/\alpha + K_D^\omega)^{-1}$ does not scale on $L_0^2(\partial D) := \lbrace f \in L^2(\partial D) : \int_{\partial D} f d\sigma = 0 \rbrace$, and $A$ and $A_{\partial D}$ scale both as $\varepsilon^2$.
Then \eqref{eq:LippSchwInt} becomes
\begin{equation*}
\left(\frac{1}{\alpha} - \frac{i \kappa_0^3 |D| \rho_0}{4 \pi} - \frac{\kappa_0^2 \, k_{0}}{2 \rho_{0}} A_{\partial D} \right) \int_{\partial D} \partial_\nu u d\sigma = \frac{1}{\alpha} \int_{\partial D} \partial_\nu u^i d\sigma + \mathcal{O}\left(\varepsilon^5 \Vert \partial_\nu u \Vert_{L^2(\partial D)} + \varepsilon^6\right),
\end{equation*}
where we have used the fact that $\int_{\partial D} (x-y) \cdot \nu(x) d\sigma(x) = \int_D \div(x-y) dx = 3 |D|.$
Then, multiplying by $\alpha$ (which scales like $\varepsilon^{-2}$), we obtain the expression of the following dominating term of $\int_{\partial D} \partial_\nu u d\sigma$,
\begin{equation*}
\left(1 - \frac{i \, \alpha \, \kappa_0^3 |D| \rho_0}{4 \pi} - \frac{\alpha \, \kappa_0^2 \, k_{0}}{2 \rho_{0}} A_{\partial D} \right) \int_{\partial D} \partial_\nu u \, d\sigma = \int_{\partial D} \partial_\nu u^i \, d\sigma + \mathcal{O}\left(\varepsilon^3 \Vert \partial_\nu u \Vert_{L^2(\partial D)} + \varepsilon^4\right).
\end{equation*}

\bigskip

In the general case of $\gamma \neq 0$, instead of identity \eqref{RkLippSch3}, we have 
\begin{equation*}
\Big(\frac{1}{\alpha} + \frac{\rho_0}{2} + (K_D^\omega)^*\Big) [\partial_\nu u] (x) - \frac{\omega^2 \gamma}{\alpha} \partial_{\nu-} \int_D G_\omega(x-y) u(y) d y  = \alpha^{-1} \,   \partial_{\nu}u^i(x).
\end{equation*}
Integrating on $\partial D$, and integrating by parts the last integral, we obtain
\begin{equation*}
\int_{\partial D} \Big(\frac{1}{\alpha} + \frac{\rho_0}{2} + (K_D^\omega)^*\Big) [\partial_\nu u](x) d\sigma(x) + \frac{\omega^2 \gamma \rho_{0}}{\alpha} \left[ \frac{\omega^{2}}{k_{0}} \int_{D}  \int_D G_\omega(x-y) u(y) dy \, dx + \int_{D} u(x) dx \right]  = \alpha^{-1} \, \int_{\partial D}  \partial_{\nu}u^i(x) d\sigma(x).
\end{equation*}
Then, with the same estimates as in \eqref{eq:LippSchwInt}, we obtain
\begin{eqnarray}\label{dint+r}
\nonumber
\left(\frac{1}{\alpha} - \frac{i \kappa_0^3 |D| \rho_0}{4 \pi} - \frac{\kappa_{0}^{2} \, k_{0}}{2 \, \rho_{0}} A_{\partial D} \right) \int_{\partial D} \partial_\nu u(x) \, d\sigma(x) &+&  \frac{\omega^2 \gamma \rho_0}{\alpha} \int_D u(x) \, dx 
= \alpha^{-1} \int_{\partial D} \partial_\nu u^i (x) d\sigma(x) \\ &-&  \frac{\omega^2 \, \gamma \, \kappa^{2}_{0}}{\alpha} \int_D \int_D G_\omega(x-y)u(y) dy dx + error,
\end{eqnarray}
where 
\begin{equation*}
error := \mathcal{O}\left(\varepsilon^{5} \Vert \partial_\nu u \Vert_{L^2(\partial D)} + \varepsilon^6\right).
\end{equation*}
Next, with help of the Cauchy-Schwartz inequality, we estimate the double volume integral as 
\begin{equation*}
\left\vert \frac{\omega^2 \, \gamma \, \kappa^{2}_{0}}{\alpha} \int_D \int_D G_\omega(x-y)u(y) dy dx \right\vert \lesssim \varepsilon^{\frac{11}{2}} \, \Vert u \Vert_{L^{2}(D)},  
\end{equation*}
then, the equation $(\ref{dint+r})$ takes the following form
\begin{equation*}
\left(\frac{1}{\alpha} - \frac{i \kappa_0^3 |D| \rho_0}{4 \pi} - \frac{\kappa_{0}^{2} \, k_{0}}{2 \, \rho_{0}} A_{\partial D} \right) \int_{\partial D} \partial_\nu u (x) \; d\sigma(x)  +  \frac{\omega^2 \gamma \rho_0}{\alpha} \int_D u (x)\; dx
= \frac{1}{\alpha} \int_{\partial D} \partial_\nu u^i (x) d\sigma(x) +r,
\end{equation*}
where
\begin{equation}\label{err}
r := \mathcal{O}\left(\varepsilon^{\frac{11}{2}} \Vert u \Vert_{L^2(D)} + \varepsilon^{5} \Vert \partial_\nu u \Vert_{L^2(\partial D)} + \varepsilon^6\right).
\end{equation}
We use $\ref{intDint}$ and the fact that $\Delta u^i = - \kappa_{0}^{2} \; u^{i}$  to obtain 
\begin{equation}\label{eq:dudnuAsympExp}
\left(\frac{1}{\alpha} - \frac{i \kappa_0^3 |D| \rho_0}{4 \pi} - \frac{\omega^{2} }{2} A_{\partial D} - \frac{\gamma k_1 \rho_0 }{\alpha \rho_1} \right) \, \int_{\partial D} \partial _\nu u  = - \frac{\omega^2 \rho_0}{\alpha k_0} \int_{D} u^i +  r = - \frac{\omega^2 \, \rho_{0}}{\alpha \, k_{0}} \vert D \vert  u^i(z) + r.  
\end{equation}
Recalling that $\beta=1/k_1-1/k_0$ and $\alpha=1/\rho_1-1/\rho_0$, then $\gamma=\beta-\alpha \rho_1/k_1 =\rho_1 / (\rho_0 k_1)- 1/k_0$, and then 
$1-\gamma k_1 \rho_0/\rho_1= \rho_0 k_1 / (\rho_1 k_0)$. 
We define the Minnaert frequency $\omega_M$ as
\begin{equation*}
\omega^2_M:= \frac{8 \, \pi \, \overline{k}_{1}}{\rho_{0} \, \mu_{\partial B}}.
\end{equation*}
Observe that $\omega^2_M$ is the dominating part of the zero, in terms of $\omega^2$, of the left hand side of $(\ref{eq:dudnuAsympExp})$.\\
To estimate the error term $r$ in $(\ref{err})$, we need the following a priori estimates.
\begin{proposition}\label{prop:uAsympExp}
For $u = u^i + u^s $, solution of $(\ref{eq:acoustic_scattering})$, it holds
\begin{equation}\label{eq:u_estimate_2}
\Vert \partial_\nu u \Vert_{L^2(\partial D)}  =  \mathcal{O}\left(\frac{\varepsilon^{2}}{\omega^2 - \omega_M^2} \right),
\end{equation}
and
\begin{equation}\label{eq:u_estimate_1}
\Vert u \Vert_{L^2(D)} = \mathcal{O}\left(\frac{\varepsilon^{\frac{3}{2}}}{\omega^2 - \omega_M^2} \right),
\end{equation}
under the condition that $\dfrac{\varepsilon}{\omega^2 - \omega_M^2}$ is small enough.
\end{proposition}
\begin{proof}
Let us indicate as $C$ a generic constant independent of $\varepsilon$.
From $(\ref{eqNS}$) we have
\begin{equation}\label{equint}
\left(I - \gamma \omega^2 N_D^\omega \right) (u) + \alpha \, S_D^\omega [\partial_\nu u] = u^{i},
\end{equation}
where $N_D^\omega$ is the Newtonian operator from $L^{2}(D)$ to $H^{2}(D)$ defined by $N_D^\omega(u)(x) := \int_{D} G_{\omega}(x-y) \, u(y) dy$. Since $\gamma = \mathcal{O}\left( 1 \right)$ and thus $\left\Vert N_D^\omega \right\Vert_{\mathcal{L}} \xrightarrow{\varepsilon \to 0} 0$, for $\varepsilon$ small enough we have that $I - \gamma \omega^2 N_D^\omega$ is invertible, so $(\ref{equint})$ takes the following form 
\begin{equation*}
u = - \alpha (I - \gamma \omega^2 N_D^\omega)^{-1} (S_D^\omega [\partial_\nu u]) + (I - \gamma \omega^2 N_D^\omega)^{-1} (u^i).  
\end{equation*}
Taking the $L^{2}$-norm in both side of the last equation and using the fact that $\left\Vert \left(I - \gamma \omega^2 N_D^\omega\right)^{-1} \right\Vert_{\mathcal{L}} \leq C$ to obtain 
\begin{equation}\label{eq:aprioribound:uL2}
\Vert u \Vert_{L^2(D)} \leq \alpha \, C \, \Vert S_D^\omega [\partial_\nu u] \Vert_{L^2(D)} + C \, \Vert u^i \Vert_{L^2(D)}.
\end{equation}
In order to finish the last estimation we need to precise how does the single layer scale. For this, by definition, we have
\begin{eqnarray}\label{chgvarsl}
\nonumber
\Big\Vert S_D^\omega (f) \Big\Vert_{L^2(D)}^2 &:=& \int_{D} \left\vert \int_{\partial D} G_\omega (x-y) \, f(y) dy \right\vert^2 dx, \quad \forall \, f \in L^{2}(\partial D) \\ 
&=& \varepsilon^5 \int_{ B} \left\vert \int_{\partial B} G_{\varepsilon \omega} (\eta - \xi) \, \tilde{f} \, (\xi) d\xi  \right\vert^2 d\eta := \varepsilon^5 \, \Big\Vert S_{B}^{\varepsilon \omega} \; (\tilde{f})  \Big\Vert_{L^2(B)}^2
\end{eqnarray}
and from the continuity of $S_B^{\varepsilon \omega}$ from $ L^2(\partial B)$ to $H^{\frac{3}{2}}(B)$ we have that
\begin{equation*}
\Big\Vert S_D^\omega (f) \Big\Vert_{L^2(D)}^2 =  \varepsilon^5 \, \Big\Vert S_{B}^{\varepsilon \omega} \; (\tilde{f})  \Big\Vert_{L^2(B)}^2 \leq  \varepsilon^5 \, C \, \big\Vert \tilde{f} \big\Vert_{L^{2}(\partial B)}^2 = \varepsilon^{3} \, C \, \big\Vert f \big\Vert_{L^{2}(\partial D)}^2,
\end{equation*}
in particular 
\begin{equation}\label{SL}
\Big\Vert S_D^\omega \left( \partial_{\nu} u \right) \Big\Vert_{L^2(D)} \leq \varepsilon^{\frac{3}{2}} \, C \, \big\Vert \partial_{\nu} u \big\Vert_{L^{2}(\partial D)}.
\end{equation}
Combining $(\ref{eq:aprioribound:uL2})$ and $(\ref{SL})$,
we obtain
\begin{equation}\label{eq:apriori:ubound1}
\Vert u \Vert_{L^2(D)}  \leq \alpha \, \varepsilon^{\frac{3}{2}} \, C \, \Vert \partial_\nu u \Vert_{L^2(\partial D)} + C \, \Vert u^i \Vert_{L^2(D)}.
\end{equation}
To manage the term $\Vert \partial_\nu u \Vert_{L^2(\partial D)}$ we use the boundary integral equation given by $(\ref{eqNDL})$, to write 
\begin{equation}\label{partialDintequa}
\partial_{\nu} u =  \alpha^{-1} \, 
\left(\frac{1}{\alpha}  + \frac{\rho_0}{2}  + (K_D^{\omega})^* \right)^{-1} \left[\partial_{\nu} u^{i} \right] + \frac{\omega^2\gamma}{\alpha} \left(\frac{1}{\alpha}  +\frac{\rho_0}{2}  + (K_D^\omega)^* \right)^{-1} \left[\partial_{\nu} N_D^\omega \left( u \right) \right] \;\; \text{on} \;\;  \partial D. 
\end{equation}
In the next, for shortness, we set
\begin{equation*}
T := \left( \frac{1}{\alpha} + \frac{\rho_0}{2}  + (K_D^\omega)^* \right)^{-1}
\end{equation*}
and we rewrite $(\ref{partialDintequa})$ as
\begin{eqnarray}\label{boundary}
\nonumber
\frac{\partial u}{\partial \nu} &=&  \frac{1}{\alpha} 
\; T \; \left[\frac{\partial u^i}{\partial \nu} - \frac{1}{\vert \partial D \vert} \int_{\partial D} \frac{\partial u^i}{\partial \nu} \right] +  \frac{1}{\vert \partial D \vert} \int_{\partial D} \frac{\partial u^i}{\partial \nu} \; \frac{1}{\alpha} \;
T \; \left[1 \right]\\
&+& \frac{\omega^2\gamma}{\alpha} \; T \;  \left[\frac{\partial N_D^\omega \left( u \right)}{\partial \nu^{-}} - \frac{1}{\vert \partial D \vert} \int_{\partial D} \frac{\partial N_D^\omega \left( u \right)}{\partial \nu^{-}} \right] + \frac{\omega^2\gamma}{\alpha} \; \frac{1}{\vert \partial D \vert} \int_{\partial D} \frac{\partial N_D^\omega \left( u \right)}{\partial \nu^{-}} \; T \; \left[1 \right].
\end{eqnarray}
Since $\frac{\rho_{0}}{2}$ is an eigenvalue of $(K_D^0)^*$ with associated eigenspace consisting of constant functions, we have the estimates
\begin{equation}\label{eq:proof:apriori:IKinv:1}
\left\Vert T \right\Vert_{\mathcal L(L^2(\partial D))} = \left\Vert \left(\frac{1}{\alpha}  + \frac{\rho_0}{2}  + (K_D^\omega)^* \right)^{-1} \right\Vert_{\mathcal L(L^2(\partial D))} \leq   C \alpha,
\end{equation}
and on the space of functions with zero average we have 
\begin{equation}\label{eq:proof:apriori:IKinv:2}
\left\Vert T \right\Vert_{\mathcal L(L_{0}^2(\partial D))} = \left\Vert \left(\frac{1}{\alpha}  + \frac{\rho_0}{2}  + (K_D^\omega)^* \right)^{-1} \right\Vert_{\mathcal L(L^2_0(\partial D))} \leq C. 
\end{equation}
 \\
Now, take the $L^{2}(\partial D)$-norm in both side of $(\ref{boundary})$, with the help of $(\ref{eq:proof:apriori:IKinv:1})$ and $(\ref{eq:proof:apriori:IKinv:2})$ we obtain   
\begin{eqnarray}\label{normpartialu}
\nonumber
\left\Vert \frac{\partial u}{\partial \nu} \right\Vert_{L^{2}(\partial D)} & \lesssim &  \alpha^{-1}  \; \left\Vert \frac{\partial u^i}{\partial \nu} - \frac{1}{\vert \partial D \vert} \int_{\partial D} \frac{\partial u^i}{\partial \nu} \right\Vert_{L_{0}^{2}(\partial D)} +  \frac{1}{\vert \partial D \vert} \left\vert \int_{\partial D} \frac{\partial u^i}{\partial \nu} \right\vert  \; \left\Vert 1 \right\Vert_{L^{2}(\partial D)} \\
&+& \alpha^{-1}  \; \left\Vert \frac{\partial N_D^\omega \left( u \right)}{\partial \nu^{-}} - \frac{1}{\vert \partial D \vert} \int_{\partial D} \frac{\partial N_D^\omega \left( u \right)}{\partial \nu^{-}} \right\Vert_{L_{0}^{2}(\partial D)} +  \frac{1}{\vert \partial D \vert} \left\vert \int_{\partial D} \frac{\partial N_D^\omega \left( u \right)}{\partial \nu^{-}} \right\vert   \;  \left\Vert 1 \right\Vert_{L^{2}(\partial D)}.
\end{eqnarray}
Obviously, we have
\begin{equation}\label{partialuinc}
\left\vert \int_{\partial D} \frac{\partial u^i}{\partial \nu} \right\vert = \left\vert \int_{D} \Delta u^i \right\vert = \vert \kappa_{0}^{2} \vert \, \left\vert \int_{D}  u^i \right\vert = \mathcal{O}\left( \varepsilon^{3} \right) 
\end{equation}
and, by the triangular inequality and the smoothness of $\partial_{\nu} u^{i}$, we obtain 
\begin{equation}\label{partialuinc-moy}
\left\Vert \frac{\partial u^i}{\partial \nu} - \frac{1}{|\partial D|} \int_{\partial D} \frac{\partial u^i}{\partial \nu} \right\Vert _{L^2_{0}(\partial D)} \lesssim \left\Vert \frac{\partial u^i}{\partial \nu} \right\Vert _{L^2_{0}(\partial D)} = \mathcal{O}\left( \varepsilon \right).
\end{equation}
We also have, recalling the definition of the Green function,
\begin{eqnarray}\label{wwm/w-wm}
\nonumber
 \int_{\partial D} \frac{\partial N_D^\omega (u)}{\partial \; \nu}(x) dx  &=& -\rho_0 \, \int_D  u(x) dx  - \omega^2 \frac{\rho_0}{k_0} \int_D \int_D G_\omega(x-y)u(y) dy dx \\ \nonumber
&=& - \rho_0 \int_D u(x) \, dx + \mathcal{O}\left(\varepsilon^{\frac{7}{2}} \; \Vert u \Vert_{L^2(D)}\right)\\ 
&\stackrel{(\ref{intDint})}{=}& \frac{k_1 \rho_0}{\omega^2 \rho_1} \int_{\partial D} \partial_\nu u(x) d\sigma(x) + \mathcal{O}\left(\varepsilon^{\frac{7}{2}} \; \Vert u \Vert_{L^2(D)}\right).
\end{eqnarray} 
We now need to estimate $\int_{\partial D} \partial_\nu u \, d\sigma.$ To do this, recalling $(\ref{eq:dudnuAsympExp})$, we have
\begin{equation}\label{kappa1/kappa0}
\int_{\partial D} \partial_\nu u \, d\sigma = \frac{\kappa_{1}^{2}}{\kappa_{0}^{2}} \; \frac{\omega^2 \omega_M^2}{(\omega^2 - \omega_M^2)} |D| u^i(z) + O \left( \frac{\varepsilon^{\frac{7}{2}} \Vert u \Vert_{L^2(D)} +\varepsilon^{2} \Vert \partial_\nu u \Vert_{L^2(\partial D)} + \varepsilon^4}{\omega^2 - \omega_M^2}  \right),
\end{equation}
which we can rewrite as
\begin{equation}\label{eq:proof:apriori:dudnuestimate}
\int_{\partial D} \partial_\nu u  \, d\sigma =  \mathcal{O} \left( \frac{\varepsilon^3}{\omega^2 - \omega_M^2} \right) + \mathcal{O}\left( \frac{\varepsilon^{\frac{7}{2}} \Vert u \Vert_{L^2(D)} +\varepsilon^{2} \Vert \partial_\nu u \Vert_{L^2(\partial D)} + \varepsilon^4}{\omega^2 - \omega_M^2}  \right). 
\end{equation}
Finally substituting this in $(\ref{wwm/w-wm})$ we obtain
\begin{equation}\label{estimate-pn-u}
\int_{\partial D} \frac{\partial N_D^\omega (u)}{\partial \nu}(x) dx = \mathcal{O}\left( \frac{\varepsilon^3}{\omega^2 - \omega_M^2} \right) + \mathcal{O}\left( \frac{\varepsilon^{\frac{7}{2}} \Vert u \Vert_{L^2(D)} +\varepsilon^{2} \Vert \partial_\nu u \Vert_{L^2(\partial D)} + \varepsilon^4}{\omega^2 - \omega_M^2}  \right).
\end{equation}
Now, we estimate the last term in the right hand side of $(\ref{normpartialu})$. For this, we simply write 
\begin{equation}\label{last-but-one}
\left\Vert \frac{\partial N_D^\omega \left( u \right)}{\partial \nu^{-}} - \frac{1}{\vert \partial D \vert} \int_{\partial D} \frac{\partial N_D^\omega \left( u \right)}{\partial \nu^{-}} \right\Vert_{L_{0}^{2}(\partial D)}  \leq  \left\Vert \frac{\partial N_D^\omega \left( u \right)}{\partial \nu^{-}}  \right\Vert_{L_{0}^{2}(\partial D)} + \frac{\Vert 1 \Vert_{L^{2}(\partial D)}}{\vert \partial D \vert} \left\vert   \int_{\partial D} \frac{\partial N_D^\omega \left( u \right)}{\partial \nu^{-}} \right\vert
\end{equation} 
and deal only with the first term since the second one is estimated by $(\ref{estimate-pn-u})$. For this, by definition and scale, we have  
\begin{eqnarray}\label{chgvarn}
\nonumber
\left\Vert \frac{\partial N_D^\omega \left( u \right)}{\partial \nu^{-}}  \right\Vert^{2}_{L^{2}(\partial D)} &:=& \int_{\partial D} \left\vert \partial \nu_{-} \int_{D} G_{\omega}(x-y) \, u(y) \, dy \right\vert^{2} d\sigma(x) \\ 
&=& \varepsilon^{4} \int_{\partial B} \left\vert \partial \nu_{-} \int_{B} G_{\omega \, \varepsilon}(\eta - \xi) \, \tilde{u}(\xi) \, d\xi \right\vert^{2} d\sigma(\eta) = \varepsilon^{4} \, \left\Vert \frac{\partial N_B^{\varepsilon \, \omega} \left( \tilde{u} \right)}{\partial \nu^{-}}  \right\Vert^{2}_{L^{2}(\partial B)}.
\end{eqnarray}
From the continuity of $N^{\varepsilon \, \omega}_{B} : L^{2}(B) \rightarrow H^{2}(B)$, we deduce that
\begin{equation}\label{normpartialNu}
\left\Vert \frac{\partial N_D^\omega \left( u \right)}{\partial \nu^{-}}  \right\Vert_{L^{2}(\partial D)} = \varepsilon^{2} \, \left\Vert \frac{\partial N_B^\omega \left( \tilde{u} \right)}{\partial \nu^{-}}  \right\Vert_{L^{2}(\partial B)} \leq \varepsilon^{2} \; C^{te} \; \Vert \tilde{u} \Vert_{L^{2}(\partial B)} = \varepsilon^{\frac{1}{2}} \; C^{te} \; \Vert u \Vert_{L^{2}(\partial D)},
\end{equation}
and plugging $(\ref{normpartialNu})$ in $(\ref{last-but-one})$ we obtain 
\begin{equation*}
\left\Vert \frac{\partial N_D^\omega \left( u \right)}{\partial \nu^{-}} - \frac{1}{\vert \partial D \vert} \int_{\partial D} \frac{\partial N_D^\omega \left( u \right)}{\partial \nu^{-}} \right\Vert_{L_{0}^{2}(\partial D)}  \lesssim  \varepsilon^{\frac{1}{2}}  \; \Vert u \Vert_{L^{2}(\partial D)} + \varepsilon^{-1} \;  \left\vert   \int_{\partial D} \frac{\partial N_D^\omega \left( u \right)}{\partial \nu^{-}} \right\vert.
\end{equation*}
Then, by $(\ref{estimate-pn-u})$, we have 
\begin{equation}\label{fff}
\left\Vert \frac{\partial N_D^\omega \left( u \right)}{\partial \nu^{-}} - \frac{1}{\vert \partial D \vert} \int_{\partial D} \frac{\partial N_D^\omega \left( u \right)}{\partial \nu^{-}} \right\Vert_{L_{0}^{2}(\partial D)}  \lesssim  \varepsilon^{\frac{1}{2}}  \; \Vert u \Vert_{L^{2}(\partial D)} + \mathcal{O}\left( \frac{\varepsilon^2}{\omega^2 - \omega_M^2} \right) + \mathcal{O}\left( \frac{\varepsilon \Vert \partial_\nu u \Vert_{L^2(\partial D)} + \varepsilon^3}{\omega^2 - \omega_M^2}  \right).
\end{equation}
Therefore, by $(\ref{partialuinc}), (\ref{partialuinc-moy}), (\ref{estimate-pn-u})$ and $(\ref{fff})$, we get 
\begin{equation}\label{eq:apriori:pdnuu}
\Vert \partial_{ \nu} u \Vert_{L^2(\partial D)} \leq \mathcal{O} \left(\frac{\varepsilon^\frac{5}{2} \Vert u \Vert_{L^2(D)}+\varepsilon \Vert \partial_\nu u \Vert_{L^2(\partial D)}}{(\omega^2 - \omega_M^2)} + \frac{\varepsilon^{2}}{\left( \omega^2 - \omega_M^2 \right)} \right)
\end{equation}
and if $\varepsilon / (\omega^2 - \omega_M^2)$ is small enough
we obtain 
\begin{equation}\label{partialDaprioriestimate}
\Vert \partial_{ \nu} u \Vert_{L^2(\partial D)} \leq \mathcal{O} \left(\frac{\varepsilon^\frac{5}{2} \Vert u \Vert_{L^2(D)}}{(\omega^2 - \omega_M^2)} + \frac{\varepsilon^{2}}{\left( \omega^2 - \omega_M^2 \right)} \right).
\end{equation}
Substituting this estimate for $\Vert \partial_{ \nu} u \Vert_{L^2(\partial D)} $ in \eqref{eq:apriori:ubound1}, to obtain 
\begin{equation*}
\Vert u \Vert_{L^2(D)}   \lesssim  \alpha \, \varepsilon^{3/2} \, \, \Vert \partial_\nu u \Vert_{L^2(\partial D)} +  \, \Vert u^i \Vert_{L^2(D)} =   \frac{\varepsilon^2 \Vert u \Vert_{L^2(D)}}{(\omega^2 - \omega_M^2)} + \frac{\varepsilon^{\frac{3}{2}}}{\left( \omega^2 - \omega_M^2 \right)}  +  \, \varepsilon^{\frac{3}{2}}.
\end{equation*}
This justify $(\ref{eq:u_estimate_1})$. Now, use $(\ref{eq:u_estimate_1})$ into  $(\ref{partialDaprioriestimate})$ to get $(\ref{eq:u_estimate_2})$.


\end{proof}
Recall $(\ref{kappa1/kappa0})$ and rewrite it, using the a priori estimate given by $(\ref{eq:u_estimate_1})$ and $(\ref{eq:u_estimate_2})$, as 
\begin{equation}\label{cv}
\int_{\partial D} \partial _\nu u \, d\sigma =  \frac{\kappa^2_1 \omega^2_M}{\left( \omega^2-\omega^2_M \right)}\vert D \vert u^i(z) +\mathcal{O}\left(\dfrac{\varepsilon^4}{\left( \omega^2 - \omega_M^2 \right)^{2}} \right).
\end{equation}
We have from $(\ref{eq:scatteredWaveAsymp})$
\begin{eqnarray*}
u^s(x) &=&  - \left(\alpha + \frac{\gamma k_1}{\rho_1}\right) G_\omega(x-z) \int_{\partial D} \frac{\partial u}{\partial \nu}(y) \, d\sigma(y) +   \mathcal{O}\left(\varepsilon^{5/2} \Vert u \Vert_{L^2(D)} + \alpha \varepsilon^{2} \Big\Vert \frac{\partial u}{\partial \nu} \Big\Vert_{L^2(\partial D)}\right)\\ 
&\stackrel{(\ref{cv})}{=}& - \left(\alpha + \frac{\gamma k_1}{\rho_1}\right) G_\omega(x-z) \left[ \frac{\kappa^2_1 \omega^2_M}{\left( \omega^2-\omega^2_M \right)}\vert D \vert u^i(z) + \mathcal{O}\left(\dfrac{\varepsilon^4}{\left( \omega^2 - \omega_M^2 \right)^{2}} \right) \right] +   O\left( \frac{\varepsilon^{2}}{\left( \omega^2 - \omega_M^2 \right)} \right)
\end{eqnarray*}
and the fact that $\alpha + \gamma k_1 / \rho_1=\rho^{-1}_1 +\mathcal{O}(1)$, and $\rho_1=\overline{\rho}_{1} \varepsilon^{2}$, we rewrite the last formula as 
\begin{eqnarray}\label{usctecase}
\nonumber
u^{s}(x) &=& - \left( \frac{1}{\rho_{1}} + \mathcal{O}(1) \right) G_\omega(x-z) \left[ \frac{\kappa^2_1 \omega^2_M}{\left( \omega^2-\omega^2_M \right)}\vert D \vert u^i(z) +O\left(\dfrac{\varepsilon^4}{\left( \omega^2 - \omega_M^2 \right)^{2}} \right) \right] +   O\left( \frac{\varepsilon^{2}}{\left( \omega^2 - \omega_M^2 \right)} \right)\\
 &=& - \frac{\omega^2 \, \omega^2_M}{ \overline{k}_1 (\omega^2-\omega^2_M)}\vert B \vert\; \varepsilon \; G_\omega(x-z)\; u^i(z) + O\left(\dfrac{\varepsilon^2}{\left( \omega^2 - \omega_M^2 \right)^{2}} \right),
\end{eqnarray}
for $x$ away from $D$ and $\varepsilon / (\omega^2 - \omega_M^2)$ small enough.
\begin{remark}\label{v=uinctecase}
Recall that, in constant coefficients case, we have $v(\cdot,\theta,\omega) = u^{i}(\cdot,\theta,\omega)$ and the equation $(\ref{scat-first-regime})$ is the same as $(\ref{usctecase})$. 
\end{remark}
Now, $(\ref{scat-first-regime})$ is proved, we deduce the corresponding far field   
\begin{equation*}
 u^{\infty}(\hat{x},\theta ,\omega) = v^{\infty}(\hat{x},\theta,\omega)- 
\frac{\omega^{2} \, \omega^{2}_{M}}{\overline{k}_{1}(\omega^{2}-\omega^{2}_{M})} \vert B \vert \; \varepsilon \; G^{\infty}_{\omega}(\hat{x}, z) \, v(z,\theta,\omega) + \mathcal{O}\left( \frac{\varepsilon^{2}}{\left( \omega-\omega_{M} \right)^{2}}\right)
\end{equation*}
and using the mixed reciprocity relation $G_{\omega}^{\infty}(\hat{x}, z) = v(-\hat{x}, z, \omega)$ we obtain 
\begin{equation*}
 u^{\infty}(\hat{x},\theta ,\omega) = v^{\infty}(\hat{x},\theta,\omega)- 
\frac{\omega^{2} \, \omega^{2}_{M}}{\overline{k}_{1}(\omega^{2}-\omega^{2}_{M})} \vert B \vert \; \varepsilon \; v(-\hat{x},z,\omega) \, v(z,\theta,\omega) + \mathcal{O}\left( \frac{\varepsilon^{2}}{\left( \omega-\omega_{M} \right)^{2}}\right).
\end{equation*}
This proves $(\ref{farfield-first-regime})$ and we finish the proof of Theorem \ref{Theorem-Using-Minnear-Resonance}.

\bigskip

\subsection{Variable coefficients}\label{sect:Minnaert:VarCoeff}
Let us suppose now that the  coefficients $ \rho_0, k_0$ vary smoothly depending on the position while in a bounded domain $\Omega$, and that they are constant outside $\Omega$. We warn the reader that we keep the same notations as in the case of constant coefficients and we shall denote by $G_\omega$ the fundamental solution satisfying \eqref{eq:defFundSol} with these variable coefficients. In this case, the Lippmann-Schwinger equation writes as 
\begin{equation}\label{LSEVarCoeff}
 u( x) -  \underset{x}{\div} \int_D \left( \frac{1}{\rho_1} - \frac{1}{\rho_0}(y) \right) G_\omega (x-y) \nabla u(y) dy - \omega^{2} \, \int_D  \left( \frac{1}{k_1} - \frac{1}{k_0}(y) \right) G_\omega (x-y)  u(y) dy = v(x).
\end{equation}
We denote by  
\begin{equation*}
\mathbf{I} := \underset{x}{\div}  \int_D \left( \frac{1}{\rho_1} - \frac{1}{\rho_0}(y) \right) G_{\omega}(x-y) \nabla u (y) d y = - \int_D \left( \frac{1}{\rho_1} - \frac{1}{\rho_0}(y) \right) \underset{y}{\nabla} G_\omega(x-y)  \cdot \nabla u(y) dy 
\end{equation*}
moreover we can write it, using integration by parts identities, as 
\begin{eqnarray*}
 \mathbf{I}  &=& \int_D G_\omega (x-y) \, \div \left( \left(\frac{1}{\rho_1} - \frac{1}{\rho_0}(y)\right) \nabla u (y) \right) dy - \int_{\partial D} G_\omega(x-y) \left(\frac{1}{\rho_1} - \frac{1}{\rho_0}(y)\right) \partial_\nu u(y) dy \\
&=& \frac{1}{\rho_1} \int_D G_\omega (x-y) \Delta u(y) dy - \int_D G_\omega (x-y) \; div  \left(\frac{1}{\rho_0}(y) \nabla u(y)\right) dy - \int_{\partial D} G_\omega(x-y) \left(\frac{1}{\rho_1} - \frac{1}{\rho_0}(y)\right) \partial_\nu u(y) dy. 
\end{eqnarray*}
Recall that, on $D$, we have $\Delta u + \kappa_{1}^{2} u = 0$ and use this to write the previous equation as 
\begin{eqnarray*}
\mathbf{I} &=&  \omega^2  \int_D \left( \frac{\rho_{1}}{k_{1} \, \rho_{0}}(y) - \frac{1}{k_{1}} \right) \; G_\omega (x-y) u(y) dy - \int_D G_\omega (x-y) \nabla  \frac{1}{\rho_0}(y) \cdot  \nabla u(y) dy  \\
&-& \int_{\partial D} G_\omega(x-y) \left(\frac{1}{\rho_1} - \frac{1}{\rho_0}(y)\right) \partial_\nu u(y) dy. 
\end{eqnarray*}
Plugging the new expression of $\mathbf{I}$ onto the Lippmann-Schwinger equation $(\ref{LSEVarCoeff})$, we obtain 
\begin{eqnarray}\label{LSED}
\nonumber
u( x)  &-&  \omega^2 \int_D \gamma(y) \; G_\omega (x-y) u(y) dy + \int_D G_\omega (x-y) \nabla  \frac{1}{\rho_0}(y) \cdot  \nabla u(y) dy \\ &+& \int_{\partial D} G_\omega(x-y) \left(\frac{1}{\rho_1} - \frac{1}{\rho_0}(y)\right) \partial_\nu u(y) dy  = v(x), \; \text{where} \; \gamma(y):=  \frac{-1}{k_{0}(y)} + \frac{\rho_{1}}{k_{1} \; \rho_{0}(y)}.
\end{eqnarray}
By taking the normal derivative from inside we deduce the corresponding integral equation on the boundary. More precisely, for $x \in \partial D$, we have
\begin{eqnarray}\label{eq:VarCoeff:LippSchwing:nabla}
\nonumber
\left[ 1 + \frac{\rho_0(x)}{2} \left( \frac{1}{\rho_1} - \frac{1}{\rho_0}(x)\right) \right] \partial_{\nu} u( x)  &-&  \omega^2 \partial_{\nu} \int_D \gamma(y) \; G_\omega (x-y) u(y) dy + \partial_{\nu} \int_D G_\omega (x-y) \nabla  \frac{1}{\rho_0}(y) \cdot \nabla u(y) dy \\ &+& \left( K_D^\omega \right)^* \left[ \left(\frac{1}{\rho_1} - \frac{1}{\rho_0}(\cdot) \right) \partial_\nu u(\cdot) \right](x)  = \partial_{\nu} v(x).
\end{eqnarray}
We use the Lippmann-Schwinger equation to derive an expression for the scattered field. To do this, for $x$ away from $D$ and $y$ such that $ \Vert y - z \Vert \sim \varepsilon$, we expand near $z$ the equation $(\ref{LSED})$ to obtain 
\begin{eqnarray}\label{u^s-varcoeff}
\nonumber
u^s(x) &=& v^{s}(x) + \omega^2 G_\omega (x-z) \, \gamma(z) \, \int_D  u(y) dy - G_\omega(x-z) \int_D \nabla \frac{1}{\rho_0}(y)  \cdot  \nabla u(y) dy \\
&-& G_\omega(x-z) \, \left( \frac{1}{\rho_1} - \frac{1}{\rho_0}(z) \right) \, \int_{\partial D}  \partial_\nu u(y) d\sigma(y) + \mathcal{O}\left(  \varepsilon^{\frac{5}{2}} \Vert u \Vert_{H^1(D)} + \frac{\varepsilon^2 }{\rho_1}  \Vert \partial_\nu u \Vert_{L^2(\partial D)} \right).
\end{eqnarray}
We know that 
\begin{eqnarray}\label{nabla-rho-0-nabla-u}
\nonumber
\int_D \nabla  \frac{1}{\rho_0}  \cdot  \nabla u dy &=& -\int_{D} \frac{1}{\rho_0}\Delta udy +\int_{\partial D}\frac{1}{\rho_0} \partial_\nu u d\sigma(y) = \frac{\omega^2 \, \rho_{1}}{k_{1}} \int_{D} \frac{1}{\rho_0}udy +\int_{\partial D}\frac{1}{\rho_0} \partial_\nu u d\sigma(y)\\
\nonumber
&=& \frac{\omega^2 \, \rho_1}{k_1 \rho_0(z)}\int_{D}udy + \frac{\omega^{2} \, \rho_{1}}{k_{1}} \nabla \frac{1}{\rho_{0}(z)} \cdot \int_{D}  (y-z) \, u (y) dy + \frac{1}{\rho_0(z)}\int_{\partial D} \partial_\nu u d\sigma(y)
\\ 
&+& \mathcal{O}\left( \varepsilon^\frac{7}{2} \Vert u \Vert_{L^2(D)}+ \varepsilon^2  \Vert \partial_\nu u \Vert_{L^2(\partial D)} \right).
\end{eqnarray}
We estimate the middle term by composing $(\ref{LSED})$ with $(\cdot - z)$ and integrating over $D$ as follow 
\begin{eqnarray*}
\int_{D} (y-z) \, u(y) dy &=& - \omega^{2} \, \int_{D} (y-z) \, N^{\omega}(\gamma \, u)(y) dy -  \, \int_{D} (y-z) \, N^{\omega}\left( \nabla \rho^{-1}_{0} \cdot \nabla u \right)(y) dy \\ 
&+& -  \, \int_{D} (y-z) \, S^{\omega}\left( \left( \frac{1}{\rho_{1}} - \frac{1}{\rho_{0}(\cdot)} \right) \, \partial_{\nu} u(\cdot) \right)(y) dy + \int_{D} (y-z) \, v(y) dy, \\
\left\vert \int_{D} (y-z) \, u(y) dy \right\vert & \lesssim & \varepsilon^{\frac{9}{2}} \, \Vert u \Vert_{L^{2}(D)} + \varepsilon^{\frac{9}{2}} \, \Vert \nabla u \Vert_{L^{2}(D)} + \varepsilon^{2} \, \Vert \partial_{\nu} u \Vert_{L^{2}(\partial D)} + \varepsilon^{4} 
\end{eqnarray*}  
and plug these estimates in $(\ref{nabla-rho-0-nabla-u})$ to obtain
\begin{eqnarray}\label{gradrho0gradu}
\nonumber
\int_D \nabla  \frac{1}{\rho_0}  \cdot  \nabla u dy &=& \frac{\omega^2 \, \rho_1}{k_1 \rho_0(z)}\int_{D}udy +  \frac{1}{\rho_0(z)}\int_{\partial D} \partial_\nu u d\sigma(y)
+ \mathcal{O}\left( \varepsilon^\frac{7}{2} \Vert u \Vert_{H^{1}(D)}+ \varepsilon^2  \Vert \partial_\nu u \Vert_{L^2(\partial D)} + \varepsilon^{4} \right)\\
& \stackrel{(\ref{intDint})}{=} & \mathcal{O}\left( \varepsilon^\frac{7}{2} \Vert u \Vert_{H^{1}(D)}+ \varepsilon^2  \Vert \partial_\nu u \Vert_{L^2(\partial D)} + \varepsilon^{4} \right).
\end{eqnarray}
Then, the equation $(\ref{u^s-varcoeff})$ takes the following form 
\begin{eqnarray}\label{ibra}
\nonumber
u^s(x)-v^{s}(x) &=& - G_\omega (x-z) \left[ - \omega^{2}  \gamma(z) \int_D  u dy + \frac{1}{\rho_1}  \int_{\partial D}  \frac{\partial u}{\partial \nu} d \sigma  \right] + \mathcal{O}\left(  \varepsilon^{\frac{5}{2}} \Vert u \Vert_{H^1(D)} + \frac{\varepsilon^2 }{\rho_1}  \Vert \partial_\nu u \Vert_{L^2(\partial D)} + \varepsilon^{4} \right) \\
 &\stackrel{(\ref{intDint})}{=}& - \frac{1}{\rho_{1}}  \,  G_\omega (x-z)  \, \int_{\partial D}  \partial_\nu u(y) \, d \sigma(y) + \mathcal{O}\left( \varepsilon^{\frac{5}{2}} \Vert u \Vert_{H^1(D)} + \frac{\varepsilon^2 }{\rho_1}  \Vert \partial_\nu u \Vert_{L^2(\partial D)} + \varepsilon^{4} \right). 
\end{eqnarray}
In the next proposition, similarly to Proposition $(\ref{prop:uAsympExp})$, we estimate the error terms appearing above.   
\begin{proposition}\label{prop:uAsympExp:VarCoeff}
For $u = u^i + u^s $, the solution of $(\ref{eq:acoustic_scattering})$, it holds
\begin{equation}\label{eq:u_estimate_2:VarCoeff}
\Vert \partial_\nu u \Vert_{L^{2}(\partial D)} = \mathcal{O}\left(\frac{\varepsilon^{2}}{\omega^2 - \omega_M^2} \right)
\end{equation}
and
\begin{equation}\label{eq:u_estimate_3:VarCoeff}
\phantom{Vide} \; \, 
\Vert u \Vert_{H^{1}(D)} = \mathcal{O} \left(\frac{\varepsilon^{\frac{1}{2}}}{\omega^2 - \omega_M^2} \right).
\end{equation}
under the condition that $\dfrac{\varepsilon}{\omega^2 - \omega_M^2}$ is small enough.
\end{proposition}
\begin{proof}
Let the two volumetric operators $N:L^2(D) \to L^2(D)$ and $M : \left( L^2(D) \right)^{3} \to L^2(D)$ defined by
\begin{equation*}
N(\varphi)(x) := \int_{D} G_{\omega}(x-y) \, \gamma(y) \, \varphi(y) \, dy \quad \text{and} \quad M(F)(x) := \int_{D} G_{\omega}(x-y) \, \nabla \frac{1}{\rho_0}(y) \cdot F(y) \, dy. 
\end{equation*}
Notice that, for $\varepsilon$ small, the operator norms of $N$ and $M$ and their derivatives all go to zero, thus for any fixed coefficient $\lambda$, the operators $I + \lambda N, I+\lambda M, I + \lambda \nabla N$ and $I+\lambda \nabla M$ are invertible, and their inverse have norm bounded by a constant $C^{te}$ independent from $\varepsilon$.
Also, we define a boundary integral operators $S$ from $L^2(\partial D)$ to  $L^2(D)$  as
\begin{equation*}
S(\psi)(x) := \int_{\partial D} G_\omega(x-y) \left(\frac{1}{\rho_1} - \frac{1}{\rho_0}(y) \right) \psi(y) \, d\sigma(y) 
\end{equation*}
Next, we estimate $u$ with $H^{1}$-norm. First, we use $(\ref{LSED})$ to write the integral equation satisfied by $u$ as follows  
\begin{equation}\label{intequau}
u = \left( I - \omega^{2} \, N \right)^{-1} \; \left[ -  M(\nabla u) -  S(\partial_{\nu} u)  +  v \right] 
\end{equation}
and taking the gradient of $(\ref{LSED})$, we get
\begin{equation*}
 \nabla u =  \left(I + \nabla M \right)^{-1} \; \left[   \omega^2 \, \nabla N(u)  - \nabla S(\partial_{\nu} u)  + \nabla v \right].
\end{equation*}
We take the $L^{2}$-norm to obtain
\begin{eqnarray}\label{normgradu}
\nonumber
\Vert \nabla u \Vert_{L^2(D)} &=&  \left\Vert \left(I + \nabla M \right)^{-1} \; \left[   \omega^2 \, \nabla N(u)  - \nabla S(\partial_{\nu} u)  + \nabla v \right] \right\Vert_{L^2(D)}\\ \nonumber
 &\leq & C^{te} \left[ \omega^2 \left\Vert \nabla N(u) \right\Vert_{L^2(D)} + \left\Vert \nabla S ( \partial_\nu u ) \right\Vert_{L^2(D)} + \left\Vert \nabla v \right\Vert_{L^2(D)} \right] \\
 &\leq & C^{te} \left(  \varepsilon \; \Vert u \Vert _{L^{2}(D)} + \frac{\varepsilon^{\frac{1}{2}}}{\rho_1} \Vert \partial_\nu u \Vert_{L^{2}(\partial D)} + \Vert \nabla v \Vert_{L^{2}(D)} \right),
\end{eqnarray}
where we used the change of  variable techniques as in  $(\ref{chgvarn})$ and $(\ref{chgvarsl}$) and remarking that the leading term\footnote{To exhibit the leading term of $S$ we should take into account the fact that $\rho_{1}$ (respectively, $\rho_{0}$) are constant function (respectively, smooth function) on the spatial variable and behaving as $\varepsilon^{-2}$ (respectively, independent on $\varepsilon$).} of $S(\psi)(\cdot)$ is given by $\rho_{1}^{-1} \, \int_{\partial D} G_{\omega}(\cdot-y) \psi(y) d\sigma(y)$. \\ 
From $(\ref{intequau})$, we also have
\begin{eqnarray*}
\Vert u \Vert_{L^{2}(D)} & \leq & \Vert (I - \omega^2 N)^{-1} \Vert _\mathcal{L}  \left( \Vert M \left( \nabla u \right) \Vert_{L^{2}(D)} + \Vert S \left( \partial_\nu u \right) \Vert_{L^{2}(D)} + \Vert v \Vert_{L^{2}(D)} \right) \\ 
& \leq & C^{te} \left( \varepsilon^{2} \Vert \nabla u \Vert_{L^{2}(D)} + \frac{\varepsilon^{\frac{3}{2}}}{\rho_1}  \Vert \partial_\nu u \Vert_{L^{2}(\partial D)} + \Vert v \Vert_{L^{2}(D)} \right).
\end{eqnarray*}
Substituting $(\ref{normgradu})$, this becomes
\begin{equation}\label{eq:proof:VarCoeff:uestimates:2}
 \Vert u \Vert_{L^{2}(D)} \leq C^{te} \left[ \frac{\varepsilon^{\frac{3}{2}}}{\rho_1}   \Vert \partial_\nu u \Vert_{L^2(\partial D)} + \varepsilon^{2} \, \Vert \nabla v \Vert + \Vert v \Vert\right] = C^{te} \; \varepsilon^{-\frac{1}{2}} \Vert \partial_\nu u \Vert_{L^2(\partial D)} + \mathcal{O}\left( \varepsilon^{\frac{3}{2}} \right). 
 \end{equation}
Putting together  $(\ref{normgradu})$ and $(\ref{eq:proof:VarCoeff:uestimates:2}$), we obtain
\begin{equation}\label{eq:proof:VarCoeff:uestimate:H1bound}
\Vert u \Vert^{2}_{H^{1}(D)} \leq C^{te} \left(\varepsilon^{-3} \; \Vert \partial_\nu u \Vert^{2}_{L^{2}(\partial D)}  + \varepsilon^{3} \right).
\end{equation}
To estimate  $\Vert \partial_\nu u \Vert_{L^{2}(\partial D)}$, we rewrite  $(\ref{eq:VarCoeff:LippSchwing:nabla})$, where we denote by $ \alpha(z) := \rho^{-1}_{1} - \rho^{-1}_{0}(z)$, as follows
\begin{equation*}
\left[ \left( \rho_{0}^{-1}(x) \, \alpha^{-1}(z) + \frac{1}{2}  \right) I +  \rho_{0}^{-1}(x) \left( K_D^\omega \right)^* \right]  \left( \partial_\nu u \right)(x) = \rho_{0}^{-1}(x) \, \alpha^{-1}(z) \, \partial_{\nu} v(x) +  \omega^2 \, \rho_{0}^{-1}(x) \, \alpha^{-1}(z) \, \partial_{\nu} N(u)(x) 
\end{equation*}
\begin{eqnarray}\label{invB}
\nonumber
\phantom{Empty} \qquad \qquad \qquad &+& \rho_{0}^{-1}(x) \, \alpha^{-1}(z) \, \partial_{\nu} M \left( \nabla u \right)(x) + \frac{\alpha^{-1}(z)}{2} \int_{0}^{1} (x-z) \cdot \nabla \frac{1}{\rho_{0}}(z+t(x-z)) \, 
 dt \, \partial_{\nu} u( x)\\
&+& \rho_{0}^{-1}(x) \, \alpha^{-1}(z) \, \left( K_D^\omega \right)^* \left[ \int_{0}^{1} (\cdot - z) \cdot \nabla \frac{1}{\rho_{0}}(z+t(\cdot - z))  \, dt \,  \partial_\nu u(\cdot) \right](x)
\end{eqnarray}
To estimate $\partial_\nu u$, we need to invert the operator $\left[ \left( \rho_{0}^{-1}(\cdot) \, \alpha^{-1}(z) + \frac{1}{2}  \right) I +  \rho_{0}^{-1}(\cdot) \left( K_D^\omega \right)^* \right]$. The next lemma gives more precisions about the well posedness of this operation.  
\begin{lemma}\label{lemma1}
Set  $\large{\text{J}}_{\omega}$ the operator defined as
\begin{eqnarray*}
\large{\text{J}}_{\omega} := L^{2}(\partial D) & \rightarrow & L^{2}(\partial D) \\ 
f  & \rightarrow & \large{\text{J}}_{\omega}(f)(x) :=  \int_{\partial D} \rho_{0}^{-1}(y) \, \frac{\partial G_{\omega}}{\partial \nu(y)}(x-y) \, f(y) \, d\sigma(y)
\end{eqnarray*}
then  $\large{\text{J}}_{\omega}^{\; *}(f)(x) = \rho_{0}^{-1}(x) \, \left(H_{D}^{w}\right)^{*}(f)(x)$,  $\large{\text{J}}_{0}(1) = -1/2$ and in addition we have
\begin{equation}\label{jw-j0}
\left(\large{\text{J}}_{\omega} - \large{\text{J}}_{0}\right)(1)(x) = - \kappa_{0}^{2} \, \frac{1}{2} \,  \int_{\partial D} \frac{\nu(y)\cdot (y-x)}{4 \, \pi \vert x - y \vert} d\sigma(y) + \mathcal{O}\left(\varepsilon^{3}\right).  
\end{equation}
\end{lemma}\label{lemma2}
\begin{proof}
See the appendix.
\end{proof}
We need also the following result. 
\begin{lemma}\label{lemma-B}
Set $B$ the operator defined from $L^{2}(\partial D)$ to $L^{2}(\partial D)$ as 
\begin{equation*}
B(f)(x) :=  \left[ \left( \rho_{0}^{-1}(x) \, \alpha^{-1}(z) + \frac{1}{2}  \right) I + \large{\text{J}}_{\omega}^{\; *} \right] (f)(x)   
\end{equation*}
then  $B$ is invertible and in addition  we have 
\begin{equation}\label{normB}
\Vert B^{-1} \Vert_{\mathcal{L}(L^{2}(\partial D))} \lesssim \alpha(z) \quad \text{and} \quad \Vert B^{-1} \Vert_{\mathcal{L}(L_{0}^{2}(\partial D))} = \mathcal{O}(1), \; \mbox{ as } \varepsilon <<1
\end{equation}
if $\overline{\rho}_{1}$ is such that $\rho_{0}(z) < \overline{\rho}_{1}$.
\end{lemma}
\begin{proof}
 See the appendix.
\end{proof}
Using Lemma \ref{lemma-B}, the equation $(\ref{invB})$, after taking the $\mathbb{L}^{2}$ norm, becomes 
\begin{eqnarray}\label{RCB}
\nonumber
\Vert \partial_\nu u \Vert_{\mathbb{L}^{2}(\partial D)} &=& \alpha^{-1}(z) \, \Bigg\Vert \, B^{-1} \, \left( \rho_{0}^{-1}  \, \partial_{\nu} v \right) + \omega^2  \, B^{-1}\left(\rho_{0}^{-1} \, \partial_{\nu} N(u)\right) +  B^{-1} \left(\rho_{0}^{-1} \, \partial_{\nu} M \left( \nabla u \right) \right) \\ \nonumber
&+& \frac{1}{2} \, B^{-1}\left(\int_{0}^{1} (\cdot -z) \cdot \nabla \frac{1}{\rho_{0}}(z+t(\cdot -z)) \, dt \, \partial_{\nu} u \right) \\
&+&  B^{-1} \left( \rho_{0}^{-1} \, \left(K_D^\omega \right)^* \left[ \int_{0}^{1} (\cdot - z) \cdot \nabla \frac{1}{\rho_{0}}(z+t(\cdot - z))  \, dt \,  \partial_\nu u(\cdot) \right] \right) \Bigg\Vert_{\mathbb{L}^{2}(\partial D)}.
\end{eqnarray}
From $(\ref{normB})$, we have
\begin{equation*}
\left\Vert B^{-1}\left(\int_{0}^{1} (\cdot -z) \cdot \nabla \frac{1}{\rho_{0}}(z+t(\cdot -z)) \, dt \, \partial_{\nu} u \right) \right\Vert \leq \alpha(z) \,  \left\vert \int_{0}^{1} (\cdot -z) \cdot \nabla \frac{1}{\rho_{0}}(z+t(\cdot -z)) \, dt \right\vert \, \Vert \partial_{\nu} u \Vert \lesssim \alpha(z) \, \varepsilon \, \Vert \partial_{\nu} u \Vert
\end{equation*}
and similarly, using the continuity of $\left( K_D^\omega \right)^*$, we obtain
\begin{eqnarray*}
\left\Vert B^{-1} \left( \rho_{0}^{-1} \, \left( K_D^\omega \right)^* \left[ \int_{0}^{1} (\cdot - z) \cdot \nabla \frac{1}{\rho_{0}}(z+t(\cdot - z))  \, dt \,  \partial_\nu u(\cdot) \right] \right) \right\Vert & \lesssim & \alpha(z) \, \left\vert \int_{0}^{1} (\cdot -z) \cdot \nabla \frac{1}{\rho_{0}}(z+t(\cdot -z)) \, dt \right\vert \, \Vert \partial_{\nu} u \Vert \\ 
& \lesssim & \alpha(z) \, \varepsilon \, \Vert \partial_{\nu} u \Vert. 
\end{eqnarray*}
The equation $(\ref{RCB})$ can be rewritten as 
\begin{equation}\label{ff}
\Vert \partial_\nu u \Vert_{\mathbb{L}^{2}(\partial D)} \lesssim \alpha^{-1}(z) \, \Bigg\Vert \, B^{-1} \, \left( \rho_{0}^{-1}  \, \partial_{\nu} u^i \right) + \omega^2  \, B^{-1}\left(\rho_{0}^{-1} \, \partial_{\nu} N(u)\right) +  B^{-1} \left(\rho_{0}^{-1} \, \partial_{\nu} M \left( \nabla u \right) \right) \Bigg\Vert_{L^{2}(\partial D)}.
\end{equation}
Next, to estimate the right hand side of $(\ref{ff})$, we proceed in three steps.
\begin{enumerate}
\item[Step 1:] Obviously, we have
\begin{equation*}
\Vert B^{-1} \, \left( \rho_{0}^{-1}  \, \partial_{\nu} v \right) \Vert \leq \left\Vert B^{-1} \, \left( \rho_{0}^{-1} \, \partial_{\nu} v - \frac{1}{\vert \partial D \vert} \int_{\partial D} \rho_{0}^{-1} \, \partial_{\nu} u^i \, d\sigma \right) \right\Vert + \frac{1}{\vert \partial D \vert} \left\vert  \int_{\partial D} \rho_{0}^{-1} \, \partial_{\nu} v \, d\sigma \right\vert  \Vert B^{-1} \, \left(1\right) \Vert
\end{equation*} 
and with help of $(\ref{partialuinc-moy})$, $(\ref{partialuinc})$ and $(\ref{normB})$, we have
\begin{equation*}\label{1stterm}
\Vert B^{-1} \, \left( \rho_{0}^{-1}  \, \partial_{\nu} v \right) \Vert_{L^{2}(\partial D)} = \mathcal{O}(1). 
\end{equation*}
\item[Step 2:] By a change of variables, as in $(\ref{chgvarn})$, and from the continuity of $N_B^{\varepsilon \omega}, M_B^{\varepsilon \omega} : L^{2}(B) \to H^{2} (B)$, we have 
\begin{equation*}
\Big\Vert \rho_{0}^{-1} \, \partial_{\nu_-} N[u] \Big\Vert_{L^{2}(\partial D)}^2  \leq \varepsilon \; C \; \Vert u \Vert_{L^{2}(D)}^{2} \quad \text{and} \quad \Big\Vert \rho_{0}^{-1} \, \partial_{\nu_-} M[ \nabla u] \Big\Vert_{L^{2}(\partial D)}^{2} \leq \varepsilon \; C \; \Vert \nabla u \Vert_{L^{2}(D)}^{2}.
\end{equation*}
Now, we use the same approach as previously by considering separately $\rho_{0}^{-1} \, F - \vert \partial D \vert^{-1} \, \int_{\partial D} \rho_{0}^{-1} \, F$ and $\vert \partial D \vert^{-1} \, \int_{\partial D} \rho_{0}^{-1} \, F$ for $F = \partial_\nu N u$ and $F = \partial_\nu M \nabla u$,  to obtain
\begin{equation*}\label{2ndterm} 
\left\Vert  B^{-1} \left[ \rho_{0}^{-1} \, \partial_\nu N(u) -  \dfrac{1}{|\partial D|} \int_{\partial D} \rho_{0}^{-1} \, \partial_\nu N(u) \, d\sigma \right] \right\Vert_{L^2(\partial D)} \leq \sqrt{\varepsilon} \; C \, \Vert u \Vert_{L^{2}(D)} 
\end{equation*}
and 
\begin{equation*}\label{3rdterm}
\left\Vert  B^{-1} \left[\rho_{0}^{-1} \, \partial_\nu M(\nabla u) -  \dfrac{1}{|\partial D|} \int_{\partial D} \rho_{0}^{-1} \,  \partial_\nu M(\nabla u) \, d\sigma \right] \right\Vert_{L^2(\partial D)} \leq \sqrt{\varepsilon} \, C \, \Vert \nabla u \Vert_{L^2(D)}.
\end{equation*}
\item[Step 3:] We deal with the term 
\begin{equation*}
\int_{\partial D} \rho_{0}^{-1}(x) \, \partial_\nu M (\nabla u)(x) \, d\sigma(x). 
\end{equation*}
Interchanging the integration and using the divergence theorem we get
\begin{eqnarray}\label{Ami3}
\nonumber
\int_{\partial D} \rho_{0}^{-1}(x) \, \partial_\nu M (\nabla u)(x) \, d\sigma &=& \int_D \nabla \frac{1}{\rho_0}(y) \cdot \nabla u(y) \int_D \underset{x}{\div}\left( \rho_{0}^{-1}(x) \,\underset{x}{\nabla}  G_\omega(x-y) \right) dx dy  \\ \nonumber
&\stackrel{(\ref{eq:defFundSol})}{=}& - \int_D \nabla \frac{1}{\rho_0}(x) \cdot \nabla u(x) dx - \omega^{2} \, \int_D \nabla \frac{1}{\rho_0}(y) \cdot \nabla u(y) \int_D  k^{-1}_0(x)  G_\omega(x-y)    dx dy \\ \nonumber
 &=&  \int_{D} \nabla \frac{1}{\rho_0(x)} \cdot \nabla u(x) dx + \mathcal{O}\left( \varepsilon^{\frac{7}{2}} \Vert \nabla u \Vert_{L^{2}(D)} \right)\\
 & \stackrel{(\ref{gradrho0gradu})}{=} & \mathcal{O}\left( \varepsilon^{\frac{7}{2}} \Vert u \Vert_{H^{1}(D)}+\varepsilon^{2} \Vert \partial_\nu u \Vert_{L^{2}(\partial D)} + \varepsilon^{4} \right).
\end{eqnarray}
The analysis of the term 
\begin{equation*}
\int_{\partial D} \rho_{0}^{-1}(x) \, \partial_\nu N (u)(x) \, d\sigma(x)
\end{equation*}
is more delicate and needs more efforts. We star by repeating the same steps of the proof of Proposition \ref{prop:uAsympExp} to obtain
\begin{eqnarray}\label{ML}
\nonumber
\int_{\partial D} \rho_{0}^{-1}(x) \partial_\nu N (u)(x) d\sigma &=& \int_D \gamma(y) \, u(y) \, \int_{\partial  D} \rho_{0}^{-1}(x) \, \partial_{\nu (x)}G_\omega(x-y) d\sigma(x) \, dy \\ \nonumber 
&=& \int_D \gamma(y) \, u(y) \, \int_{D} \underset{x}{\div}\left(\rho_{0}^{-1}(x) \, \underset{x}{\nabla} G_\omega(x-y) \right) dx \, dy \\ \nonumber
&\stackrel{\ref{eq:defFundSol}}{=}& - \int_D \gamma(x) \, u(x) \,dx - \omega^{2} \int_D \gamma(y) \, u(y) \,\int_D k^{-1}_0(x) G_\omega(x-y)  \, dx \,  dy  \\
&=& -  \gamma(z) \int_{D} u (x) dx - \nabla \gamma(z) \cdot \int_{D} (x-z) \, u(x) dx + \mathcal{O}\left( \varepsilon^{\frac{7}{2}} \Vert u \Vert_{L^{2}(D)} \right). 
\end{eqnarray}

To estimate the second term in the right hand side of the last equation we use $(\ref{LSED})$ to get 
\begin{eqnarray*}
\int_{D} (x-z) u(x) dx &=& - \omega^{2} \int_{D} (x-z) N(u)(x) dx + \int_{D} (x-z) M(\nabla u) (x) dx \\ 
&+& \int_{D} (x-z) S(\partial_{\nu} u)(x) dx + \int_{D} (x-z) v(x) dx \\
\left\vert \int_{D} (x-z) u(x) dx \right\vert & \lesssim & \varepsilon^{\frac{9}{2}} \, \Vert u \Vert + \varepsilon^{\frac{9}{2}} \, \Vert \nabla u \Vert + \varepsilon^{2} \, \Vert \partial_{\nu} u \Vert + \mathcal{O}\left( \varepsilon^{4} \right). 
\end{eqnarray*}
Finally, the equation $(\ref{ML})$, with the help of $(\ref{intDint})$, takes the following form
\begin{equation}\label{ce}
\int_{\partial D} \rho_{0}^{-1} \, \partial_\nu N (u) \, d\sigma = \frac{ \gamma(z) k_1}{\omega^2 \rho_1}\int_{\partial D}\partial_\nu u \, d\sigma + \mathcal{O}\left(\varepsilon^{\frac{7}{2}} \Vert u \Vert_{L^{2}(D)} + \varepsilon^{\frac{9}{2}} \Vert \nabla u \Vert_{L^{2}(D)}+ \varepsilon^{2} \Vert \partial_{\nu} u \Vert_{L^{2}(\partial D)} + \varepsilon^{4} \right).
\end{equation}
To finish with the estimation of $(\ref{ce})$, we need to estimate the integral of $\partial_{\nu} u$ over $\partial D$. For this, the identity $(\ref{eq:VarCoeff:LippSchwing:nabla})$ can be rewritten as
\begin{eqnarray*}
\frac{1}{2} \, \left( \rho_{1}^{-1} + \rho_0^{-1}(x) \right) \partial_\nu u(x) &+&  \large{\text{J}}_{\omega}^{\, *} \left[ \left( \rho^{-1}_{1} - \rho_0^{-1}(\cdot) \right) \partial_\nu u (\cdot) \right](x) - \omega^2 \rho_0^{-1}(x) \, \partial_{\nu-}  N(u)(x) \\ 
&+& \rho_0^{-1}(x) \, \partial_{\nu-} M \left(\nabla u \right) (x)  = \rho_0^{-1}(x) \, \partial_\nu v(x).
\end{eqnarray*}
We integrate this equation over $\partial D$ and use $(\ref{ce})$ , $(\ref{Ami3})$ and $\large{\text{J}}_{0}(1) = -1/2$ to obtain 
\begin{eqnarray*}
\frac{k_{1}}{\rho_{1} \, k_{0}}  \, \int_{\partial D}  \partial_\nu u(x) \, d\sigma(x) &+&  \int_{\partial D} \left( \large{\text{J}}_{\omega} - \large{\text{J}}_{0} \right) (1)(x) \, \left[ \left( \rho^{-1}_{1} - \rho_0^{-1}(x) \right) \partial_\nu u (x) \right] d\sigma(x) \\
&=& \int_{\partial D} \rho_0^{-1}(x) \, \partial_\nu v(x) d\sigma(x) + \mathcal{O}\left( \varepsilon^{\frac{7}{2}} \Vert u \Vert_{H^{1}(D)}+\varepsilon^{2} \Vert \partial_\nu u \Vert_{L^{2}(\partial D)} + \varepsilon^{4} \right).
\end{eqnarray*}
and from $(\ref{jw-j0})$ we obtain 
\begin{eqnarray*}
\frac{k_{1}}{\rho_{1} \, k_{0}}  \, \int_{\partial D}  \partial_\nu u(x) \, d\sigma(x) &-&  \frac{\kappa_{0}^{2} \, \alpha(z)}{2}  \int_{\partial D}  \,  \int_{\partial D} \frac{\nu(y)\cdot (y-x)}{4 \, \pi \vert x - y \vert} d\sigma(y) \,  \partial_\nu u (x)  d\sigma(x) \\
&=& \int_{\partial D} \rho_0^{-1}(x) \, \partial_\nu v(x) d\sigma(x) + \mathcal{O}\left( \varepsilon^{\frac{7}{2}} \Vert u \Vert_{H^{1}(D)}+\varepsilon^{2} \Vert \partial_\nu u \Vert_{L^{2}(\partial D)} + \varepsilon^{4} \right).
\end{eqnarray*}

Using the same notations as in $(\ref{defA})$ and the fact that $v$ satisfy the equation $(\ref{uincequa})$ we write 
\begin{eqnarray*}
\left( \frac{k_{1}}{\rho_{1} \, k_{0}} - \frac{\omega^{2} \, \alpha(z)}{2 \, \rho_{0}}  \,A_{\partial D} \, \right) \int_{\partial D}  \partial_\nu u(x) \, d\sigma(x) 
&=& - \frac{\omega^{2}}{k_{0}(z)} \int_{D}  v(x) dx + \mathcal{O}\left( \varepsilon^{\frac{7}{2}} \Vert u \Vert_{H^{1}(D)}+\varepsilon^{2} \Vert \partial_\nu u \Vert_{L^{2}(\partial D)} + \varepsilon^{4} \right).\\ 
& + & \frac{\omega^{2} \, \alpha(z)}{2 \, \rho_{0}}  \,  \int_{\partial D}  \, ( A(x) - A_{\partial D})  \,  \partial_\nu u (x)  d\sigma(x) 
\end{eqnarray*}
and using the definition of $A_{\partial D}$ and the estimation given in $(\ref{ll})$ we obtain  
\begin{eqnarray}\label{estint}
\nonumber
\left( \frac{k_{1}}{\rho_{1} \, k_{0}} - \frac{\omega^{2} \, \alpha(z) \, \rho_{0}}{8 \, \pi \, k_{0}}  \,\mu_{\partial D} \, \right) \int_{\partial D}  \partial_\nu u (x)\, d\sigma(x) & = & - \frac{\omega^{2}}{k_{0}(z)} \int_{D}  v(x) dx \\ 
&+& \mathcal{O}\left( \varepsilon^{\frac{7}{2}} \Vert u \Vert_{H^{1}(D)}+\varepsilon^{2} \Vert \partial_\nu u \Vert_{L^{2}(\partial D)} + \varepsilon^{4} \right). 
\end{eqnarray}
We derive the corresponding estimate as in $(\ref{eq:proof:apriori:dudnuestimate})$:
\begin{equation*}
\int_{\partial D} \partial_\nu u d\sigma = \mathcal{O}\left( \frac{\varepsilon^3}{\omega^2-\omega_M^2} \right) + \mathcal{O} \left( \frac{\varepsilon^{\frac{7}{2}} \Vert u \Vert_{H^{1}(D)}+ \varepsilon^{2} \Vert \partial_\nu u \Vert_{L^{2}(\partial D)}+ \varepsilon^{4}}{\omega^2-\omega_M^2} \right),
\end{equation*}
where now 
\begin{equation*}
\omega^{2}_{M} := \omega^{2}_{M}(z) :=  \frac{8 \, \pi \, \overline{k}_{1}}{\rho_{0}(z) \, \mu_{\partial B}}. 
\end{equation*}
Back substituting these results into $(\ref{ce})$ we obtain 
\begin{equation*}\label{AM}
\int_{\partial D} \rho_{0}^{-1}(x) \, \partial_\nu N(u)(x) \, d\sigma(x) = \frac{- \gamma(z) \, k_1  \, 8 \pi}{(8 \pi \, k_{1} \, - \, \omega^{2} \, \rho_{0} \, \mu_{\partial D})}    \int_{D}  v(x) dx + \mathcal{O}\left( \varepsilon^{\frac{7}{2}} \Vert u \Vert_{H^{1}(D)}+ \varepsilon^{2} \Vert \partial_{\nu} u \Vert_{L^{2}(\partial D)} + \varepsilon^{4} \right).
\end{equation*}
\end{enumerate}
With the same calculations as for \eqref{eq:apriori:pdnuu}, we derive the estimate
\begin{equation*}
\Vert \partial_\nu u \Vert_{L^{2}(\partial D)}  =    \mathcal{O}\left( \frac{ \varepsilon^{\frac{5}{2}} \Vert u \Vert_{H^{1}(D)} +\varepsilon \Vert \partial_\nu u \Vert_{L^{2}(\partial D)} + \varepsilon^{2} }{\omega^2 - \omega_{M}^2}  \right).
\end{equation*}
Hence if $\varepsilon / (\omega^2 - \omega_M^2)$ is small enough, we have
\begin{equation*}
\Vert \partial_\nu u \Vert_{L^{2}(\partial D)}  =    \mathcal{O}\left( \frac{ \varepsilon^{\frac{5}{2}} \Vert u \Vert_{H^{1}(D)} + \varepsilon^{2} }{\omega^2 - \omega_{M}^2}  \right).
\end{equation*}
Finally from $(\ref{eq:proof:VarCoeff:uestimate:H1bound})$,  we obtain $(\ref{eq:u_estimate_2:VarCoeff})$ and $(\ref{eq:u_estimate_3:VarCoeff})$.
\end{proof}
\begin{remark}
A straightforward calculation, from $(\ref{normgradu})$, $(\ref{eq:proof:VarCoeff:uestimates:2})$ and the estimation of $\Vert \partial_{\nu} u \Vert_{L^{2}(\partial D)}$, allows to deduce that 
\begin{equation*}
\Vert u \Vert_{L^{2}(D)} = \mathcal{O}\left( \frac{\varepsilon^{\frac{3}{2}}}{\omega^{2}-\omega^{2}_{M}} \right)       \quad \text{and} \quad \Vert \nabla u \Vert_{L^{2}(D)} = \mathcal{O}\left( \frac{\varepsilon^{\frac{1}{2}}}{\omega^{2}-\omega^{2}_{M}} \right).
\end{equation*}
\end{remark}
To compute the first order approximation of $\int_{\partial D} \partial_\nu u \, d\sigma$, we rewrite $\ref{estint}$ as 
\begin{eqnarray*}
 \frac{\rho_{0} \, \mu_{\partial B}}{8 \pi \, \overline{\rho}_{1} \, k_{0}} (\omega_{M}^{2}-\omega^{2})  \, \int_{\partial D}  \partial_\nu u (x)\, d\sigma(x) & = & - \frac{\omega^{2}}{k_{0}(z)} \int_{D}  v(x) dx - \frac{\omega^{2} \, \varepsilon^{2} \, \mu_{\partial B}}{8 \, \pi \, k_{0}}  \int_{\partial D}  \partial_\nu u (x)\, d\sigma(x) \\ 
&+& \mathcal{O}\left( \varepsilon^{\frac{7}{2}} \Vert u \Vert_{H^{1}(D)}+\varepsilon^{2} \Vert \partial_\nu u \Vert_{L^{2}(\partial D)} + \varepsilon^{4} \right). 
\end{eqnarray*}
Successively, on the right hand side,  by Cauchy-Schwartz inequality, we estimate the term containing $\partial_{\nu} u$   as $\mathcal{O}\left(\varepsilon^{3} \, \Vert \partial_{\nu} u \Vert_{L^{2}(\partial D)}\right)$, we use the Taylor expansion of $v$ near the center and Proposition \ref{prop:uAsympExp:VarCoeff}, to obtain
\begin{equation*}
 \int_{\partial D}  \partial_\nu u (x)\, d\sigma(x)  =   \frac{\omega^{2} \, 8\pi \, \overline{\rho}_{1}}{\rho_{0} \, \mu_{\partial B}( \omega^{2} - \omega^{2}_{M} )} \,  v(z) \, \vert D \vert  + \mathcal{O}\left( \frac{\varepsilon^{4}}{ \left( \omega^{2}-\omega_{M}^{2} \right)^{2}}  \right). 
\end{equation*}
Plugging this estimation in $(\ref{ibra})$, and with help of the Proposition \ref{prop:uAsympExp:VarCoeff}, we obtain, for $x$ away from $D$,
\begin{eqnarray}
u^s(x)  &=& v^{s}(x)  - \,  G_\omega (x-z)  \,  \frac{\omega^{2} \, \omega^{2}_{M}}{\overline{k}_{1} \, ( \omega^{2} - \omega^{2}_{M} )} \,  v(z) \, \vert B \vert \, \varepsilon + \mathcal{O}\left( \frac{\varepsilon^{2}}{ \left( \omega^{2}-\omega^{2}_{M} \right)^{2}} \right), 
\end{eqnarray}
as it was done for $(\ref{usctecase})$.\\ Using the mixed reciprocity relation, we derive the expansion of the associated far fields as it was mentioned in $(\ref{farfield-first-regime})$. 

\section{Proof of Theorem \ref{Theorem-Using-Newtonian-Resonances}}

To avoid additional lengthy computations, we provide the detailed proof in the case the background is homogeneous. The case of inhomogeneous background can be handled following the steps described in the previous section. 
\bigskip

The starting point is, again, the system of the integral equations:
\begin{equation}\label{eq:ReworkLippmannSchwinger-chap4}
  u(x) -  \gamma \omega^2  \int_D G_\omega(x-y)u(y) dy + \alpha \int_{\partial D} G_\omega (x-y) \, \partial _{\nu} u(y) \, d \sigma(y) = u^i(x), \mbox{ in } D ,
\end{equation}
\begin{equation}\label{eq:ReworkLippmannSchwinger2-chap4}
\left[ \left(1 + \frac{\alpha \rho_0}{2} \right) I + \alpha ( K_D^{\omega})^*\right] \left[ \partial_{\nu} u \right](x) - \gamma \omega^2   \partial_{\nu-} \int_D G_\omega(x-y)u(y) dy  = \partial_{\nu} u^{i}(x), \mbox{ on } \partial D.
\end{equation}
Notice that due to the scaling of $\rho_1$ and $k_1$, in this regime, we have 
\begin{equation}\label{2ndregime}
\gamma \sim \varepsilon^{-2} \;\; \text{while} \;\; \alpha
\sim \varepsilon^{j},\; j>0, \;\; \text{as} \;\; \varepsilon \to 0
\end{equation}
where we recall that $\gamma := \beta - \alpha \, \rho_{1} \; k_{1}^{-1}, \; \alpha := \rho_{1}^{-1}-\rho_{0}^{-1}$ and $\beta := k_{1}^{-1}-k_{0}^{-1}$.\\
The strategy of the proof is quite similar to the previous section. Indeed, we first provide the a priori estimation of both $u$ and $\partial_{\nu} u$ and then derive the dominating term of the expansion of the scattered fields. The main difference is the fact that the regimes, fixed by the contrasts of the mass densities and bulk moduli, are different. 

\subsection{A priori estimation}
We start with the equation $(\ref{eq:ReworkLippmannSchwinger2-chap4})$, i.e   
\begin{equation}\label{equa+43}
\left[ \left(1 + \frac{\alpha \rho_0}{2} \right) I + \alpha ( K_D^{\omega})^*\right] \left[ \partial_{\nu}u \right](x) = \gamma \omega^2   \partial_{\nu} N^{\omega}(u)(x)  + \partial_{\nu}u^{i}(x).
\end{equation}
As $\alpha$ is small, see ($\ref{2ndregime}$), then $\left[ \left(1 + \frac{\alpha \rho_0}{2} \right) I + \alpha ( K_D^{\omega})^*\right]^{-1}$ exists. Taking the inverse in both sides of $(\ref{equa+43})$ we obtain 
\begin{equation*}\label{equa43}
\partial_{\nu} u = \gamma \omega^2  \left[ \left(1 + \frac{\alpha \rho_0}{2} \right) I + \alpha ( K_D^{\omega})^*\right]^{-1} \left[ \partial_{\nu}N^{\omega}(u) \right] + \left[ \left(1 + \frac{\alpha \rho_0}{2} \right) I + \alpha ( K_D^{\omega})^*\right]^{-1} \left[\partial_{\nu} u^{i} \right]
\end{equation*}
and then by taking the $L^{2}(\partial D)$-norm, we obtain 
\begin{eqnarray*}\label{B}
\nonumber
\left\Vert \frac{\partial u}{\partial \nu} \right\Vert_{L^{2}(\partial D)} & \leq &  \frac{\vert \gamma \, \omega^2 \vert}{\alpha} \left\Vert \left[ \left(\frac{1}{\alpha} + \frac{\rho_0}{2} \right) I +  ( K_D^{\omega})^*\right]^{-1} \right\Vert_{\mathcal{L}\left({L^{2}(\partial D)} \right)} \left\Vert \partial_{\nu} N^{\omega}(u) \right\Vert_{L^{2}(\partial D)} \\  
&+& \frac{1}{\alpha} \left\Vert \left[ \left(\frac{1}{\alpha} + \frac{\rho_0}{2} \right) I +  ( K_D^{\omega})^*\right]^{-1} \right\Vert_{\mathcal{L}\left({L^{2}(\partial D)} \right)} \; \left\Vert \frac{\partial u^i}{\partial \nu} \right\Vert_{L^{2}(\partial D)}.
\end{eqnarray*}
Since
\begin{equation*}
 \Bigg\Vert  \Bigg[\Big(\frac{1}{\alpha} + \frac{\rho_{0}}{2}  \Big) \, I +  (K_{D}^{\omega})^{\star} \Bigg]^{-1}\Bigg\Vert_{\mathcal{L}} \lesssim  \frac{1}{dist\left(\frac{1}{\alpha} + \frac{\rho_{0}}{2} ;\sigma\left( (K_{D}^{\omega})^{\star}\right) \right) } = \frac{1}{\Big\vert \frac{1}{\alpha} + \frac{\rho_{0}}{2} - \frac{1}{2}\Big\vert} \simeq \alpha 
\end{equation*}
and
\begin{equation*}
\left\Vert \partial_{\nu} \, N^{\omega}(u) \,\right\Vert_{L^{2}(\partial D)} \leq C^{te} \; \varepsilon^{\frac{1}{2}} \; \Vert u \Vert_{L^{2}(D)}
\end{equation*}
then we obtain 
\begin{equation}\label{E}
\left\Vert \partial_{\nu} u \right\Vert_{L^{2}(\partial D)} \lesssim \; \gamma \; \varepsilon^{\frac{1}{2}} \; \Vert u \Vert_{L^{2}(D)} \; +  \;\; \left\Vert \partial_{\nu} u^{i} \right\Vert_{L^{2}(\partial D)}.
\end{equation}
Next, we consider the equation $(\ref{eq:ReworkLippmannSchwinger-chap4})$, i.e  
\begin{equation}\label{Mrm}
  u(x) -  \gamma \omega^2  N^{\omega}(u)(x) = - \alpha \, S\left(\partial_{\nu}u \right)(x) + u^i(x) ,
\end{equation}
as
\begin{eqnarray*}
N^{\omega}(u)(x) &:=& \int_{D} G_{\omega}(x,y) \, u(y) \, dy = \int_{D} \Gamma_{\omega}(x,y) \, u(y) \, dy + \int_{D} \left(G_{\omega}- \Gamma_{\omega}\right)(x,y) \, u(y) \, dy \\ 
&=& \int_{D} \Gamma_{0}(x,y) \, e^{i \, \kappa_{0} \, \vert x-y \vert} \, u(y) \, dy + \int_{D} \left(G_{\omega}- \Gamma_{\omega}\right)(x,y) \, u(y) \, dy \\
&=& \int_{D} \Gamma_{0}(x,y) \, \sum_{n \geq 0} \frac{\left( i \, \kappa_{0} \, \vert x-y \vert \right)^{n}}{n!}  \, u(y) \, dy + \int_{D} \left(G_{\omega}- \Gamma_{\omega}\right)(x,y) \, u(y) \, dy \\ 
& = & N^{0}(u)(x) +  \sum_{n \geq 1}  \int_{D} \Gamma_{0}(x,y) \, \frac{\big( i \, \kappa_{0} \, \vert x-y \vert \big)^{n}}{n!}  \, u(y) \, dy + \int_{D} \left(G_{\omega}- \Gamma_{\omega}\right)(x,y) \, u(y) \, dy.
\end{eqnarray*}
the equation $(\ref{Mrm})$ takes the following form 
\begin{eqnarray}\label{zz}
\nonumber
u(x) - \gamma \, \omega^{2} \, N^{0}(u) (x) &=& - \alpha \, S\left(\partial_{\nu} u \right)(x) + \, u^{i}(x) + \gamma \, \omega^{2} \int_{D} \left(G_{\omega}- \Gamma_{\omega}\right)(x,y) \, u(y) \, dy \\
&+& \gamma \, \omega^{2} \, \rho_{0} \sum_{n \geq 1} \left(i \, \kappa_{0}\right)^{n} \int_{D}  \frac{\vert x-y \vert^{n-1}}{n!}  \, u(y) \, dy.  
\end{eqnarray}
We take the $L^{2}(D)$-norm in both sides of the last equation to obtain
\begin{eqnarray}\label{equa1}
\nonumber
\left\Vert u - \gamma \, \omega^{2} \, N^{0}(u) \right\Vert  & \leq & \alpha \, \left\Vert S\left(\partial_{\nu} u \right) \right\Vert + \left\Vert u^{i} \right\Vert + \vert \gamma \, \omega^{2} \vert \, \Vert u \Vert \, \left[\int_{D} \int_{D} \left\vert G_{\omega}- \Gamma_{\omega}\right\vert^{2} (x,y) \, dx \, dy \right]^{\frac{1}{2}} \\
&+& \vert \gamma \, \omega^{2} \, \rho_{0} \vert  \sum_{n \geq 1} \vert \kappa_{0}  \vert^{n} \, \left\Vert \int_{D}  \frac{\vert \cdot -y \vert^{n-1}}{n!}  \, u(y) \, dy \right\Vert. 
\end{eqnarray}
Recall, see for instance $(\ref{SL})$, that
\begin{equation*}
\left\Vert S^{\omega}\left(\partial_{\nu} u \right) \right\Vert_{L^{2}(D)} \leq C^{te} \; \varepsilon^{\frac{3}{2}} \; \left\Vert \partial_{\nu} u \right\Vert_{L^{2}(\partial D)}
\end{equation*}
and 
\begin{equation*}
\sum_{n \geq 1} \vert \kappa_{0}  \vert^{n} \, \left\Vert \int_{D}  \frac{\vert \cdot -y \vert^{n-1}}{n!}  \, u(y) \, dy \right\Vert \leq \Vert u \Vert \, \varepsilon^{3} \,  \sum_{n \geq 1} \vert \kappa_{0}  \vert^{n} \, \frac{\varepsilon^{n-1}}{n!}  = \mathcal{O}\left(\Vert u \Vert \, \varepsilon^{3} \right). 
\end{equation*}
In the appendix $(\ref{lemmaproof})$ we prove that the function $\left(G_{\omega} - \Gamma_{\omega} \right)$ are bounded, then  
\begin{equation}\label{normG-G}
\left[\int_{D} \int_{D} \left\vert G_{\omega}- \Gamma_{\omega}\right\vert^{2} (x,y) \, dx \, dy \right]^{\frac{1}{2}} = \mathcal{O}\left( \varepsilon^{3} \right).
\end{equation}
Then $(\ref{equa1})$ becomes
\begin{equation}\label{equa411}
\left\Vert u - \gamma \, \omega^{2} \, N^{0}(u) \right\Vert_{L^{2}\left( D \right)}   \lesssim \alpha \, \varepsilon^{\frac{3}{2}} \, \left\Vert \partial_{\nu} u \right\Vert_{L^{2}\left(\partial D \right)} + \left\Vert u^{i} \right\Vert_{L^{2}\left( D \right)} + \varepsilon \, \Vert u \Vert_{L^{2}\left( D \right)}.
\end{equation}
Plugging $(\ref{E})$ in $(\ref{equa411})$, we obtain 
\begin{equation}\label{norm(I-N)}
\left\Vert u - \gamma \, \omega^{2} \, N^{0}(u) \right\Vert_{L^{2}\left( D \right)}   \lesssim \alpha \, \varepsilon^{\frac{3}{2}} \, \left\Vert \partial_{\nu} u^{i} \right\Vert_{L^{2}(\partial D)} + \left\Vert u^{i} \right\Vert_{L^{2}\left( D \right)} + \left( \alpha + \varepsilon \right) \, \Vert u \Vert_{L^{2}\left( D \right)}.
\end{equation}
Let $\left( \lambda^{D}_n, \rho_{0} \, e^{D}_n \right)_{n \in \mathbb{N}}$ the eigensystem of the Newtonian operator $N^{0}$ which is positive, compact and selfadjoint on $L^{2}(D)$. Using this basis, the left hand side of $(\ref{norm(I-N)})$ can be computed as\footnote{Where $<\cdot;\cdot>$ stands for the $L^{2}(D)$ inner product.}
\begin{eqnarray}\label{seqbb}
\nonumber
\left\Vert u - \gamma \, \omega^{2} \, N^{0}(u) \right\Vert^{2}_{L^{2}(D)} &=& \sum_{n} \left\vert < u - \gamma \, \omega^{2}  \, N^{0}(u) ; e_{n} > \right\vert^{2} = \sum_{n} \left\vert < u;e_{n} > \right\vert^{2} \; \left\vert 1 - \gamma \, \omega^{2} \, \rho_{0} \, \lambda^{D}_{n} \right\vert^{2} \\
&=& \left\vert < u;e_{n_{0}} > \right\vert^{2} \; \left\vert 1 - \gamma \, \omega^{2} \, \rho_{0} \, \lambda^{D}_{n_{0}} \right\vert^{2} + \sum_{n \neq n_{0}} \left\vert < u;e_{n} > \right\vert^{2} \; \left\vert 1 - \gamma \, \omega^{2}  \, \rho_{0} \, \lambda^{D}_{n} \right\vert^{2}. 
\end{eqnarray}
Next, we choose $\omega^{2}$ such that $\Big\vert 1 - \gamma \, \omega^{2} \, \rho_{0} \, \lambda^{D}_{n_{0}} \Big\vert \sim \varepsilon^{h}$, which implies that 
\begin{equation}\label{defwn0}
\omega^{2} := \frac{1 \pm \varepsilon^{h}}{\gamma \, \rho_{0} \, \lambda^{D}_{n_{0}}}, \;\; \text{or} \;\; (\omega^{2}-\omega^{2}_{n_{0}}) \sim \varepsilon^{h} \;\; \text{where} \;\; \omega^{2}_{n_{0}} := \frac{1}{\gamma \, \rho_{0} \, \lambda^{D}_{n_{0}}}
\end{equation}
since, using the definition of $\gamma$, the scale of $\rho_{1}$, $k_{1}$ and the fact that $\lambda^{D}_{n_{0}} = \lambda^{B}_{n_{0}} \, \varepsilon^{2}$, we have 
\begin{equation*}
\omega_{n_{0}}^{-2} := \gamma \, \rho_{0} \, \lambda^{D}_{n_{0}} = \left( \frac{-\rho_{0}}{k_{0}} + \frac{\rho_{1} \, \varepsilon^{-2}}{\overline{k}_{1}} \right) \,  \, \lambda^{B}_{n_{0}} \, \varepsilon^{2} = \mathcal{O}(1).
\end{equation*}
Since we choose $\omega^{2}$ close to $\omega^{2}_{n_{0}}$ we deduce that, for $n \neq n_{0}$, the sequence  $\left\vert 1 - \gamma \, \omega^{2} \, \rho_{0} \, \lambda^{D}_{n} \right\vert^{2}$ is bounded below. Then, if we set $\sigma := \underset{n \neq n_{0}}{\inf} \left\vert 1 - \gamma \, \omega^{2} \, \rho_{0} \, \lambda^{D}_{n} \right\vert^{2}$, the  equation $(\ref{seqbb})$ becomes
\begin{eqnarray*}
\left\Vert u - \gamma \, \omega^{2} \, N^{0}(u) \right\Vert^{2}_{L^{2}(D)} & \geq & \left\vert < u;e_{n_{0}} > \right\vert^{2} \; \left\vert 1 - \gamma \, \omega^{2} \, \rho_{0} \, \lambda^{D}_{n_{0}} \right\vert^{2} + \sigma \, \sum_{n \neq n_{0}} \left\vert < u;e_{n} > \right\vert^{2}. 
\end{eqnarray*}
From the previous equation, with help of $(\ref{norm(I-N)})$, we deduce that  
\begin{equation}\label{Z1}
\left\vert < u;e_{n_{0}} > \right\vert^{2}  \lesssim \left\vert 1 - \gamma \, \omega^{2} \, \rho_{0} \, \lambda_{n_{0}} \right\vert^{-2} \left[ \alpha \, \varepsilon^{\frac{3}{2}} \, \left\Vert \partial_{\nu} u^{i} \right\Vert_{L^{2}(\partial D)} + \left\Vert u^{i} \right\Vert_{L^{2}\left( D \right)} +  \left( \alpha + \varepsilon \right) \, \Vert u \Vert_{L^{2}\left( D \right)}\right]^{2}
\end{equation}
and 
\begin{equation}\label{Z2}
\sum_{n \neq n_{0}} \left\vert <u;e_{n}> \right\vert^{2} \lesssim \sigma^{-1} \, \left[ \alpha \, \varepsilon^{\frac{3}{2}} \, \left\Vert \partial_{\nu} u^{i} \right\Vert_{L^{2}(\partial D)} + \left\Vert u^{i} \right\Vert_{L^{2}\left( D \right)} + \left( \alpha +  \varepsilon \right) \, \Vert u \Vert_{L^{2}\left( D \right)}\right]^{2}.
\end{equation}
\label{cdth}
Now, if we sum $(\ref{Z1})$ and $(\ref{Z2})$, we obtain  
\begin{eqnarray*}
\left\Vert u \right\Vert_{L^{2}(D)}^{2} &:=& \left\vert <u;e_{n_{0}}> \right\vert^{2} + \sum_{n \neq n_{0}} \left\vert <u;e_{n}> \right\vert^{2} \\
& \lesssim & \left( \sigma^{-1} + \left\vert 1 - \gamma \, \omega^{2} \, \rho_{0} \, \lambda^{D}_{n_{0}} \right\vert^{-2} \, \right) \left[ \alpha^{2} \, \varepsilon^{3} \, \left\Vert \partial_{\nu} u^{i} \right\Vert^{2}_{L^{2}(\partial D)} + \left\Vert u^{i} \right\Vert^{2}_{L^{2}\left( D \right)} + \left( \alpha + \varepsilon \right)^{2} \, \Vert u \Vert^{2}_{L^{2}\left( D \right)}\right] \\
& \lesssim & \frac{1}{\left\vert 1 - \gamma \, \omega^{2} \, \rho_{0} \, \lambda^{D}_{n_{0}} \right\vert^{2}} \, \left[ \alpha^{2} \, \varepsilon^{3} \, \left\Vert \partial_{\nu} u^{i} \right\Vert^{2}_{L^{2}(\partial D)} + \left\Vert u^{i} \right\Vert^{2}_{L^{2}\left( D \right)} + \left( \alpha + \varepsilon \right)^{2} \, \Vert u \Vert^{2}_{L^{2}\left( D \right)}\right]
\end{eqnarray*}
then
\begin{eqnarray*}
\Vert u \Vert^{2}_{L^{2}(D)} \left(1 - \frac{\left( \alpha + \varepsilon \right)^{2}}{ \left\vert 1 - \gamma \, \omega^{2} \, \rho_{0} \, \lambda^{D}_{n_{0}} \right\vert^{2}} \right) & \lesssim &  \frac{1}{\left\vert 1 - \gamma \, \omega^{2} \, \rho_{0} \, \lambda_{n_{0}} \right\vert^{2}} \; \left[ \alpha^{2} \, \varepsilon^{3} \, \left\Vert \partial_{\nu} u^{i} \right\Vert^{2}_{L^{2}(\partial D)} + \left\Vert u^{i} \right\Vert^{2}_{L^{2}\left( D \right)} \right]. 
\end{eqnarray*}
We choose $\alpha$ such that $1 - \left( \alpha + \varepsilon \right)^{2} \, \left\vert 1 - \gamma \, \omega^{2} \, \rho_{0} \, \lambda^{D}_{n_{0}} \right\vert^{-2}$ is uniformly bounded from below. For this, we see that 
\begin{equation*}
1 -  \left( \alpha + \varepsilon \right)^{2} \, \left\vert 1 - \gamma \, \omega^{2} \, \rho_{0} \, \lambda^{D}_{n_{0}} \right\vert^{-2} \, \simeq 1 - \varepsilon^{-2h} \, \left( \alpha + \varepsilon \right)^{2}. 
\end{equation*}
Take $\alpha \sim \varepsilon^{j}$, where  $j > 0$, then 
\begin{equation*}
1 -  \left( \alpha + \varepsilon \right)^{2} \, \left\vert 1 - \gamma \, \omega^{2} \, \rho_{0} \, \lambda^{D}_{n_{0}} \right\vert^{-2} \, \simeq 1 - \varepsilon^{2\left(\min(1,j)-h\right)}. 
\end{equation*} 
This implies, if $\,h < \min(j,1)$, the boundedness from below of $1 - \left( \alpha + \varepsilon \right)^{2} \, \left\vert 1 - \gamma \, \omega^{2} \, \rho_{0} \, \lambda^{D}_{n_{0}} \right\vert^{-2}$. Now, under the condition
\begin{equation}\label{hj1}
h < \min(j,1)
\end{equation}
we get an estimation of $\Vert u \Vert$ with respect to $\Vert u^{i} \Vert$ and $\Vert \partial_{\nu} u^{i} \Vert$ as follows 
\begin{equation}\label{fdd}
\Vert u \Vert^{2}_{L^{2}(D)}   \lesssim   \frac{1}{\left\vert 1 - \gamma \, \omega^{2} \, \rho_{0} \, \lambda_{n_{0}} \right\vert^{2}} \; \left[\varepsilon^{3+2j} \, \left\Vert \partial_{\nu} u^{i} \right\Vert^{2}_{L^{2}(\partial D)} + \left\Vert u^{i} \right\Vert^{2}_{L^{2}\left( D \right)} \right]. 
\end{equation}
We plug $(\ref{fdd})$ into $(\ref{E})$ to obtain an estimation of $\Vert \partial_{\nu} u \Vert$  with respect to $\Vert u^{i} \Vert$ and $\Vert \partial_{\nu} u^{i} \Vert$. Precisely, we obtain
\begin{equation}\label{norder}
\left\Vert \partial_{\nu} u \right\Vert^{2}_{L^{2}(\partial D)}  \lesssim  \; \left\Vert \partial_{\nu} u^{i} \right\Vert^{2}_{L^{2}(\partial D)} + \varepsilon^{-3-2h} \;  \left\Vert u^{i} \right\Vert^{2}_{L^{2}(D)}.
\end{equation}


\subsection{Estimation of the scattered field}
\
\\
We write the integral equation $(\ref{zz})$ and we develop the incident field $u^{i}$ near the center $z$ to obtain 
\begin{eqnarray*}
\left( I - \gamma \, \omega^{2} \, N^{0}\right)(u) (x) &=& - \alpha \, S^{\omega}\left(\partial_{\nu} u \right)(x) + \, u^{i}(z) + \int_{0}^{1}(x-z)\cdot \nabla u^{i}(z+t(x-z)) \; dt  \\
&+& \gamma \, \omega^{2} \int_{D} \left(G_{\omega}- \Gamma_{\omega}\right)(x,y) \, u(y) \, dy  + \gamma \, \omega^{2} \, \rho_{0} \, i \, \kappa_{0} \, \int_{D}  u(y) \, dy \\
&+& \gamma \, \omega^{2} \, \rho_{0} \sum_{n \geq 2} \left(i \, \kappa_{0}\right)^{n} \int_{D}  \frac{\vert x-y \vert^{n-1}}{n!}  \, u(y) \, dy.  
\end{eqnarray*}
Next, successively, we set $W$ to be $ W := \left(I - \gamma \, \omega^{2}  \, N^{0} \right)^{-1}(1),$ apply the self adjoint operator $\left(I - \gamma \, \omega^{2} \, N^{0} \right)^{-1}$ in both sides of the previous equation and integrate over the domain $D$ to obtain 
\begin{eqnarray}\label{intu}
\nonumber
\int_{D} u(x) dx  &=& u^{i}(z) \; \int_{D} W(x) dx +  \int_{D} W(x) \, \int_{0}^{1} (x-z)\centerdot \nabla u^{i}(z+t(x-z)) \, dt \, dx \\ \nonumber
&-& \alpha \, \int_{D} W(x) \, S^{\omega}\left( \partial_{\nu} u \right)(x) dx + \gamma \, \omega^{2} \, \rho_{0} \, i \, \kappa_{0} \, \int_{D} W(x) dx \, \int_{D} u(y) \, dy  \\ \nonumber
& + & \gamma \, \omega^{2} \, \int_{D} W(x) \int_{D} \left(G_{\omega} - \Gamma_{\omega} \right)(x,y) \, u(y) dy \, dx \\
&+& \gamma \, \omega^{2} \, \rho_{0} \, \sum_{n \geq 2} \frac{\left(i \, \kappa_{0} \right)^{n}}{n!} \, \int_{D} W(x) \,   \, \int_{D}  \vert x - y \vert^{n-1} \, u(y) \, dy \, dx.
\end{eqnarray}
We keep only, on the right hand side, the first term and we estimate the others as an error. For this, we need first an a priori estimation of $\int_{D} W(x) dx$ and $\Vert W \Vert_{L^{2}(D)}$. To do this, we have 
\begin{equation*}
\int_{D} e^{D}_{n}(x) \, dx =\int_{D} 1 \, e^{D}_{n}(x) \, dx=   \int_{D} \left( I - \gamma \, \omega^{2} \, N^{0} \right)(W)(x) \,  e_{n}(x) dx = <W;e^{D}_{n}> \left(1 - \gamma \, \omega^{2} \, \rho_{0} \, \lambda^{D}_{n} \right)
\end{equation*} 
which implies that\,  
$<W;e^{D}_{n}> = <1;e^{D}_{n}> \; \left(1 - \gamma \, \omega^{2} \, \rho_{0} \, \lambda^{D}_{n} \right)^{-1}$\,
and then  
\begin{eqnarray}\label{Adz}
\nonumber
\int_{D} W(x) \, dx &=& \sum_{n} <W,e^{D}_{n}> \;  \int_{D} e^{D}_{n}(x) \, dx =  \sum_{n} \frac{\left( \int_{D} e^{D}_{n}(x) \, dx \right)^{2}}{\left( 1 - \gamma \, \omega^{2} \, \rho_{0} \, \lambda^{D}_{n} \right)} \\
&=& \frac{\left( \int_{D} e^{D}_{n_{0}}(x) \, dx \right)^{2}}{\left( 1 - \gamma \, \omega^{2} \, \rho_{0} \, \lambda^{D}_{n_{0}} \right)} + \sum_{n \neq n_{0}} \frac{\left( \int_{D} e^{D}_{n}(x) \, dx \right)^{2}}{\left( 1 - \gamma \, \omega^{2} \, \rho_{0} \, \lambda^{D}_{n} \right)}. 
\end{eqnarray}
Obviously, we have 
\begin{equation*}
\sum_{n \neq n_{0}} \frac{\left( \int_{D} e^{D}_{n}(x) \, dx \right)^{2}}{\left( 1 - \gamma \, \omega^{2} \, \rho_{0} \, \lambda^{D}_{n} \right)} = \mathcal{O}\left( \varepsilon^{3} \right).
\end{equation*}
Then 
\begin{equation}\label{intW}
\int_{D} W(x) \, dx \sim \mathcal{O}\left(\varepsilon^{3-h}\right).
\end{equation}
Similarly,
\begin{equation*}
\Vert W \Vert^{2}_{L^{2}(D)} = \sum_{n} \vert <W,e^{D}_{n}> \vert^{2} =
\frac{\left\vert <1;e^{D}_{n_{0}}> \right\vert^{2}}{\left\vert 1 - \gamma \, \omega^{2} \, \rho_{0} \, \lambda^{D}_{n_{0}} \right\vert^{2}} + \sum_{n \neq n_{0}} \frac{\left\vert <1;e^{D}_{n}> \right\vert^{2}}{\left\vert 1 - \gamma \, \omega^{2} \, \rho_{0} \, \lambda^{D}_{n} \right\vert^{2}}.
\end{equation*}
Then  
\begin{equation}\label{normW}
\Vert W \Vert_{L^{2}(D)} \sim \mathcal{O}\left( \varepsilon^{\frac{3}{2}-h} \right).
\end{equation}
Now, we are ready to estimate the error parts of $(\ref{intu})$. To achieve this, we split it as follows
\begin{enumerate}  
\item[$\ast$] Estimation of $I_{1} :=  \int_{D} W(x) \, \int_{0}^{1} (x-z) \cdot \nabla u^{i}(z+t(x-z)) \, dt \, dx$.\\
By the Cauchy-Schwartz inequality, we have
\begin{equation*}
\vert I_{1}  \vert  \leq  \Vert W \Vert_{L^{2}(D)} \; \left\Vert \int_{0}^{1} (\cdot -z)\cdot \nabla u^{i}(z+t(\cdot -z)) \, dt \right\Vert_{L^{2}(D)} = \mathcal{O}\left( \varepsilon^{4-h} \right).
\end{equation*}
\newline 
\item[$\ast$] Estimation of $I_{2} := \alpha \, \int_{D} W(x) \, S^{\omega}\left( \partial_{\nu} u \right)(x) dx.$\\
By applying the Cauchy-Schwartz inequality and the continuity of the single layer, see for instance the inequality $(\ref{SL})$, we obtain 
\begin{equation*}
\vert I_{2} \vert \leq \alpha \, \Vert W \Vert_{L^{2}(D)} \, \Vert S^{\omega} \left( \partial_{\nu} u \right) \Vert_{L^{2}(D)} \lesssim \varepsilon^{3-h+j} \, \Vert \partial_{\nu} u  \Vert_{L^{2}(\partial D)}
\end{equation*}
and recall that, see $(\ref{norder})$, we have $
\left\Vert \partial_{\nu} u \right\Vert^{2}_{L^{2}(\partial D)}  \lesssim  \; \left\Vert \partial_{\nu} u^{i} \right\Vert^{2}_{L^{2}(\partial D)} + \varepsilon^{-3-2h} \;  \left\Vert u^{i} \right\Vert^{2}_{L^{2}(D)}$ then 
\begin{equation*}
\vert I_{2}  \vert^{2}  \lesssim  \varepsilon^{6+2j-2h} \; \left[ \left\Vert \partial_{\nu} u^{i} \, \right\Vert^{2}_{L^{2}(\partial D)} + \varepsilon^{-3-2h} \;  \left\Vert u^{i} \right\Vert^{2}_{L^{2}(D)} \right].    
\end{equation*}
Since the incident field is smooth we have 
$\left\Vert \partial_{\nu} u^{i} \right\Vert^{2}_{L^{2}(\partial D)} = \mathcal{O}\left( \varepsilon^{2} \right)$ and $ \Vert u^{i} \Vert^{2}_{L^{2}(D)} = \mathcal{O}\left( \varepsilon^{3}\right)$.\\
With this $I_{2} = \mathcal{O}(\varepsilon^{3+j-2h})$. 
\newline 
\item[$\ast$] Estimation of  $I_{3} := \gamma \, \omega^{2} \, \rho_{0} \, i \, \kappa_{0} \, \int_{D} W(x) dx \, \int_{D} u(y) \, dy$.  \\
A straightforward application of $(\ref{intW})$ and the Cauchy-Schwartz inequality allows to deduce 
\begin{equation*}
\vert I_{3} \vert   =  \left\vert \gamma \, \omega^{2} \, \rho_{0}  \, \kappa_{0} \, \int_{D} W dx  \, \int_{D} u \, dy \right\vert \lesssim  \varepsilon^{-2+3-h} \, \Vert 1 \Vert_{L^{2}(D)} \, \Vert u \Vert_{L^{2}(D)} = \varepsilon^{\frac{5}{2}-h} \, \Vert u \Vert_{L^{2}(D)} 
\end{equation*}
then, with help of $(\ref{fdd})$, we obtain 
\begin{equation*}
\vert I_{3} \vert^{2} \lesssim \varepsilon^{5-2h} \, \Vert u \Vert^{2}   \lesssim \varepsilon^{5-2h} \, \Big\vert 1 - \gamma \, \omega^{2} \, \rho_{0} \, \lambda^{D}_{n_{0}} \Big\vert^{-2}  \left[  \varepsilon^{3+2j} \, \left\Vert \partial_{\nu} u^{i} \right\Vert^{2}_{L^{2}(\partial D)} + \Vert u^{i} \Vert^{2}_{L^{2}(D)} \right].
\end{equation*}
Recall again that $\left\Vert \partial_{\nu} u^{i} \right\Vert^{2}_{L^{2}(\partial D)} = \mathcal{O}\left( \varepsilon^{2} \right)$ and $ \Vert u^{i} \Vert^{2}_{L^{2}(D)} = \mathcal{O}\left( \varepsilon^{3}\right)$, then we deduce that $I_{3} = \mathcal{O}\left( \varepsilon^{4-2h} \right)$.
\newline  
\item[$\ast$] Estimation of $I_{4} := \left\vert \gamma \, \omega^{2} \, \rho_{0} \, \underset{n \geq 2}{\sum} \, \dfrac{\left(i \, \kappa_{0} \right)^{n}}{n!} \, \int_{D} W(x) \, \int_{D}   \vert x - y \vert^{n-1} \, u(y) \, dy \, dx \right\vert.$ \\ 
We have
\begin{equation*} 
\left\vert I_{4} \right\vert  \lesssim  \varepsilon^{-2} \, \sum_{n \geq 2} \frac{\varepsilon^{n-1}}{n!} \, \int_{D} \vert W(x) \vert \, dx \, \int_{D}  \vert u \vert (y) \, dy \leq  \varepsilon \, \Vert W \Vert \, \Vert u \Vert \, \sum_{n \geq 2} \frac{\varepsilon^{n-1}}{n!} = \varepsilon^{\frac{7}{2}-h} \, \Vert u \Vert,  
\end{equation*}
and again, by $(\ref{fdd})$, we obtain  
\begin{equation*}
\left\vert I_{4} \right\vert^{2}  \lesssim   \varepsilon^{7-2h}  \,  \Big\vert 1 - \gamma \, \omega^{2} \, \rho_{0} \, \lambda^{D}_{n_{0}} \Big\vert^{-2}  \left[  \varepsilon^{3+2j} \, \left\Vert \partial_{\nu} u^{i} \right\Vert^{2}_{L^{2}(\partial D)} + \Vert u^{i} \Vert^{2}_{L^{2}(D)} \right] = \varepsilon^{10-4h}. 
\end{equation*}
Finally $I_{4} = \mathcal{O}\left(\varepsilon^{5-2h}\right)$.
\newline
\item[$\ast$] Estimation of $I_{5} := \gamma \, \omega^{2} \, \int_{D} W(x) \, \int_{D} \left(G_{\omega} - \Gamma_{\omega} \right)(x,y) \, u(y) \, dy \, dx.$ \\   
We use Cauchy-Schwartz inequality to obtain 
\begin{equation*}
\vert I_{5} \vert \leq \varepsilon^{-2} \, \Vert W \Vert_{L^{2}(D)} \, \Vert u \Vert_{L^{2}(D)} \, \left[ \int_{D} \int_{D} \vert G_{\omega} - \Gamma_{\omega} \vert^{2}(x,y) \, dy \, dx \right]^{\frac{1}{2}}.
\end{equation*}
We use simultaneously $(\ref{normW})$, $(\ref{normG-G})$ and $(\ref{fdd})$ to get
\begin{equation*}
\vert I_{5} \vert^{2} \leq \varepsilon^{5-2h} \,  \Big\vert 1 - \gamma \, \omega^{2} \, \rho_{0} \, \lambda^{D}_{n_{0}} \Big\vert^{-2}  \left[  \varepsilon^{3+2j} \, \left\Vert \partial_{\nu} u^{i} \right\Vert^{2}_{L^{2}(\partial D)} + \Vert u^{i} \Vert^{2}_{L^{2}(D)} \right] = \varepsilon^{8-4h}
\end{equation*} 
then  $I_{5} = \mathcal{O}\left( \varepsilon^{4-2h} \right)$.
\end{enumerate}
We deduce, from  $(\ref{intu})$, that  
\begin{equation}\label{az}
\int_{D} u(x) \, dx = u^{i}(z) \; \int_{D} W(x) \; dx + \sum_{n=1}^{5} I_{n}  
 =  u^{i}(z) \, \int_{D} W(x) \, dx +  \mathcal{O}\left( \varepsilon^{3-h + \min(1;j)-h} \right). 
\end{equation}
Remark that the last formula makes sense because $\int_{D} W dx$ is of order $\varepsilon^{3-h}$, see the estimation $(\ref{intW})$, and from the condition $(\ref{hj1})$ we know that $h < \min(1,j)$ or equivalently $\min(1,j) - h >0$. \\
\newline
We know, from $(\ref{eq:ReworkLippmannSchwinger-chap4}$), that the scattered field is given by 
\begin{equation*}
u^{s}(x) = \gamma \; \omega^{2} \; \int_{D} G_\omega(x-y) \, u(y) \, dy - \alpha \; \int_{\partial D} G_\omega(x-y) \; \partial_{\nu}u(y) \, d\sigma(y).
\end{equation*}  
For $x$ away from $D$, we expand $G_\omega$ near $z$ to obtain 
\begin{eqnarray}\label{us}
\nonumber
u^{s}(x) &=& \gamma \; \omega^{2} \; G_\omega(x-z) \int_{D}  u(y) \, dy - \alpha \;G_\omega(x-z) \int_{\partial D} \, \partial_{\nu}u(y) \, d\sigma(y) \\
\nonumber 
& + & \gamma \, \omega^{2} \, \int_{D} \, \int_{0}^{1} \, (y-z)\cdot \underset{y}{\nabla} G_\omega(x-z-t(y-z)) \, dt \; u(y) \, dy \\ 
&-& \alpha \; \int_{\partial D} \int_{0}^{1} \, (y-z)\cdot \underset{y}{\nabla} G_\omega(x-z-t(y-z)) \, dt \; \partial_{\nu}u(y) \, d\sigma(y). 
\end{eqnarray}
We need to estimate the two last terms of the previous equation.
\newline
\begin{enumerate}
\item[$\ast$] Estimation of $ B_{1} := \gamma \, \omega^{2} \, \int_{D} \, \int_{0}^{1} \, (y-z)\cdot \underset{y}{\nabla} G_\omega(x-z-t(y-z)) \, dt \; u(y) \, dy.$ \\
We have
\begin{eqnarray*}
\left\vert B_{1} \right\vert & \leq & \varepsilon^{-2} \, \Vert u \Vert_{L^{2}(D)} \; \left[ \int_{D} \,\left\vert \int_{0}^{1} \, (y-z)\cdot \underset{y}{\nabla} G_\omega(x-z-t(y-z)) \, dt \right\vert^{2} \, dy \right]^{\frac{1}{2}} \\ 
& \leq & \varepsilon^{-2} \, \Vert u \Vert_{L^{2}(D)} \; \left[ \int_{D}  \, \vert y-z \vert^{2}  \, dy \right]^{\frac{1}{2}} \lesssim \varepsilon^{-2} \; \varepsilon^{\frac{3}{2}-h} \; \varepsilon^{\frac{5}{2}} = \mathcal{O}\left(\varepsilon^{2-h} \right).
\end{eqnarray*}
\newline
\item[$\ast$] Estimation of $B_{2} :=   \alpha \; \int_{\partial D} \int_{0}^{1} \, (y-z)\cdot \underset{y}{\nabla} G_\omega(x-z-t(y-z)) \, dt \, \partial_{\nu}u(y) \, d\sigma(y).$\\
We have
\begin{eqnarray*}
\left\vert B_{2} \right\vert & \leq & \varepsilon^{j} \; \left\Vert \partial_{\nu} u \right\Vert_{L^{2}(\partial D)} \; \left[\int_{\partial D} \left\vert \int_{0}^{1} \, (y-z)\cdot \underset{y}{\nabla} G_\omega(x-z-t(y-z)) \, dt\right\vert^{2} d\sigma(y) \right]^{\frac{1}{2}} \\ 
& \leq & \varepsilon^{j} \; \left\Vert \partial_{\nu} u \right\Vert_{L^{2}(\partial D)} \; \left[\int_{\partial D} \, \vert y-z \vert^{2} d\sigma(y) \right]^{\frac{1}{2}} \lesssim \varepsilon^{j} \; \varepsilon^{-h} \; \varepsilon^{2} =  \mathcal{O}\left( \varepsilon^{2+j-h} \right).
\end{eqnarray*} 
\end{enumerate} 
Taking into account the estimation of $B_{1}$ and $B_{2}$ we rewrite  the formula $(\ref{us})$ as  
\begin{eqnarray*}\label{usa}
\nonumber
u^{s}(x) &=& \gamma \; \omega^{2} \; G_\omega(x-z) \int_{D}  u(y) \, dy - \alpha \;G_\omega(x-z) \int_{\partial D} \partial_{\nu}u(y) \, d\sigma(y) + \mathcal{O}\left( \varepsilon^{2-h} \right) \\
&\stackrel{\ref{intDint}}{=}& \beta \; \omega^{2} \; G_\omega(x-z) \int_{D}  u(y) \, dy + \mathcal{O}\left( \varepsilon^{2-h} \right).
\end{eqnarray*}
Now, we use the expression of $\int_{D} u (x) dx$ given in the formula $(\ref{az})$ to deduce that
\begin{eqnarray*}
u^{s}(x) &=& \beta \; \omega^{2} \; G_\omega(x-z) \left[ u^{i}(z) \; \int_{D} W(x) dx + \mathcal{O}\left( \varepsilon^{3-2h + \min(1;j)} \right) \right] + \mathcal{O}\left( \varepsilon^{2-h} \right) \\
 &=& \beta \; \omega^{2} \; G_\omega(x-z) \; u^{i}(z) \; \int_{D} W(x) dx + \mathcal{O}\left( \varepsilon^{1-2h + \min(1;j)} \right).  
\end{eqnarray*}
Plugging the estimation of $\int_{D} W(x) \, dx$, given in formula $(\ref{Adz})$, we obtain  
\begin{eqnarray*}
u^{s}(x) &=& \beta \; \omega^{2} \; G_\omega(x-z) \; u^{i}(z) \; \left[\frac{\left( \int_{D} e^{D}_{n_{0}}(x) \, dx \right)^{2}}{\left(1 - \gamma \, \omega^{2} \, \rho_{0} \, \lambda^{D}_{n_{0}} \right)} + \mathcal{O}\left( \varepsilon^{3} \right) \right] + \mathcal{O}\left( \varepsilon^{1-2h + \min(1;j)} \right).
\end{eqnarray*}
Finally, recalling the value of $\beta := k_{1}^{-1} - k_{0}^{-1}$, we have
\begin{eqnarray*}
u^{s}(x) &=& \frac{1}{k_1} \; \omega^{2} \; G_\omega(x-z) \; u^{i}(z) \; \frac{\left( \int_{D} e^{D}_{n_{0}}(x) \, dx \right)^{2}}{\left(1 - \gamma \, \omega^{2} \, \rho_{0} \, \lambda^{D}_{n_{0}} \right)} +  \mathcal{O}\left( \varepsilon + \varepsilon^{1-2h + \min(1;j)} \right)\\
&\stackrel{\ref{defwn0}}{=}& \frac{-1}{k_1} \; \; G_\omega(x-z) \; u^{i}(z) \; \frac{\omega^{2}_{n_{0}} \, \omega^{2} }{\left(\omega^{2}  - \omega^{2}_{n_{0}} \right)} \left( \int_{D} e^{D}_{n_{0}}(x) \, dx \right)^{2} +  \mathcal{O}\left( \varepsilon + \frac{\varepsilon^{1+\min(1;j)}}{\left(\omega^{2}  - \omega^{2}_{n_{0}} \right)^{2}} \right).
\end{eqnarray*}
We recall that $k_{1} := \overline{k}_{1} \; \varepsilon^{2}$, $\int_{D} e^{D}_{n}(x) \, dx = \varepsilon^{3} \int_{B} e^{B}_{n}(x) \, dx$ and rewrite the last equation as 
\begin{equation*}
u^{s}(x) = \frac{-1}{\overline{k}_1} \;G_\omega(x-z) \; u^{i}(z) \, \varepsilon \; \frac{\omega^{2}_{n_{0}} \, \omega^{2} }{\left(\omega^{2}  - \omega^{2}_{n_{0}} \right)} \left( \int_{B} e^{B}_{n_{0}}(x) \, dx \right)^{2} +  \mathcal{O}\left( \varepsilon + \frac{\varepsilon^{1+\min(1;j)}}{\left(\omega^{2}  - \omega^{2}_{n_{0}} \right)^{2}} \right).
\end{equation*}
Then same remark as $(\ref{v=uinctecase})$ holds for this case 
and justify the equations $(\ref{Approximation-second-regime-scattered-field})$, $(\ref{Approximation-second-regime-farfield})$ and Theorem \ref{Theorem-Using-Newtonian-Resonances}. 

\section{appendix}\label{lemmaproof}
This section is devoted to justify Lemma \ref{lemma1} and Lemma \ref{lemma-B}. \\
 
Let $\Gamma_\omega(x,y) :=  \rho_0(z)\dfrac{e^{i  \kappa_{0}  |x-y|}}{4 \pi   |x-y|}$ be the fundamental solution of the equation (\ref{eq:defFundSol}) with constant coefficients $\rho_0(z)$ and $k_0$ satisfying the radiation conditions at infinity.
By expanding in $|x-y|$ we have
\begin{equation*}\label{gradGw-gradG0}
\underset{y}{\nabla} \Gamma_\omega(x,y) = \underset{y}{\nabla} \Gamma_{0}(x,y) - \rho_0(z)\; (y-x) \left( \frac{ \kappa_{0}^2}{8 \pi |x-y|} + \frac{i   \kappa_{0}^3 }{12 \pi} + O(|x-y|) \right).
\end{equation*}
Then, from $(\ref{eq:defFundSol})$ and since $ \underset{x}{\Delta} \Gamma_\omega(x,y) + \kappa_0^2 \, \Gamma_\omega(x,y) = -  \rho_0(z) \, \delta_y(x)$, integrating $G_\omega(x,y) - \Gamma_\omega(x,y)$ against the Dirac delta $-\delta_y(x)$, we have for any ball $B_R$ of large radius
\begin{eqnarray}\label{eq:Go-G0_int_repr}
\nonumber
\left( G_\omega - \Gamma_\omega \right)(x,y) &=& - \int_{B_R} \left(\frac{1}{\rho_0(t)} - \frac{1}{\rho_0(y)} \right) \nabla G_\omega(t,x) \cdot \nabla \Gamma_\omega (t,y) d t \\
&+& \omega^2 \int_{B_R} \left(\frac{1}{k(t)} - \frac{1}{k(y)} \right) G_\omega(x,t) \Gamma_\omega(t,y) dt  + \int_{\partial B_R} \left(\frac{1}{\rho_0(t)} - \frac{1}{\rho_0(y)} \right) \partial_\nu G_\omega (x,y) \Gamma_\omega(t,y) dt. 
\end{eqnarray}
The first integral in $(\ref{eq:Go-G0_int_repr})$ is $\mathcal{O}\left( \log |x-y| \right)$ because
\begin{equation*}
\int_{B_R} \left(\frac{1}{\rho_0(t)} - \frac{1}{\rho_0(y)} \right) \nabla G_\omega(t,x) \cdot \nabla \Gamma_\omega (t,y) d t = \int_{B_R} |t-y| \dfrac{1}{|t-x|^2}  \dfrac{1}{|t-y|^2} dt + r(t,y),
\end{equation*} 
where $r$ is a bounded function;
the second integral is bounded due to the smoothness of $k$; and the last integral is also bounded if we take $R$ large enough.
Thus by the divergence theorem and the properties of the fundamental solutions $G_\omega$ and $\Gamma_{\omega}$ we have 
\begin{eqnarray}\label{ha}
\nonumber
\int_D \nabla \cdot \Big( \rho(y)^{-1} \nabla G_\omega(x,y) &-& \rho(z)^{-1} \nabla \Gamma_\omega(x,y) \Big) dy = \omega^2 \int_D \left( k(y)^{-1} G_\omega(x,y) dy -  k(z)^{-1} \Gamma(x,y)\right) dy \\
&=& \omega^2 \int_D (k(y)^{-1} - k(z)^{-1}) G_\omega(x,y) dy +\omega^2  \int_D k(z)^{-1} \left(G_\omega - \Gamma_\omega \right)(x,y) dy = \mathcal{O}(\varepsilon^3) .
\end{eqnarray}
\
\newline

In order to prove Lemma \ref{lemma1} and Lemma \ref{lemma-B}, we start by recalling the definition of the  operator $\large{\text{J}}_{\omega}$
\begin{eqnarray*}
\large{\text{J}}_{\omega} := L^{2}(\partial D) & \rightarrow & L^{2}(\partial D) \\ 
f  & \rightarrow & \large{\text{J}}_{\omega}(f)(x) :=  \int_{\partial D} \rho_{0}^{-1}(y) \, \frac{\partial G_{\omega}}{\partial \nu(y)}(x-y) \, f(y) \, d\sigma(y).
\end{eqnarray*}

Let $<\cdot;\cdot>$ the $L^{2}(\partial D)$ inner product and let $f$ and $g$ two functions on $L^{2}(\partial D)$. We have
\begin{eqnarray*}
<\large{\text{J}}_{\omega}(f);g> &:=& \int_{\partial D} \large{\text{J}}_{\omega}(f)(x) \, g(x) \, d\sigma(x) = \int_{\partial D} g(x) \int_{\partial D} \rho_{0}^{-1}(y) \, \partial_{\nu(y)} G_{\omega}(x-y) \, f(y) \, d\sigma(y) \, d\sigma(x)\\ 
&=& \int_{\partial D} f(x) \rho_{0}^{-1}(x)\int_{\partial D} g(y) \partial_{\nu(x)} G_{\omega}(x,y)  d\sigma(y) d\sigma(x) \\ &=& \int_{\partial D} f(x) \rho_{0}^{-1}(x)\left( K_{D}^{\omega}\right)^{*}\left(g\right)(x)d\sigma(x) = <f;\rho_{0}^{-1}\left( K_{D}^{\omega}\right)^{*}\left(g\right)>. 
\end{eqnarray*}
This proves that $\large{\text{J}}_{\omega}^{\; *}(f)(\cdot) = \rho_{0}^{-1}(\cdot) \, \left(K_{D}^{w}\right)^{*}(f)(\cdot)$.\\
Next, by definition, we have
\begin{equation*}
\large{\text{J}}_{0}(1)(x) := \left(K_{D}^{0}\right)(\rho_{0}^{-1})(x)  = \int_{\partial D} \partial_{\nu(y)} G_{0}(x,y) \, \rho_{0}^{-1}(y) \, d\sigma(y) = \int_{D} \underset{y}{\nabla} \cdot \left( \rho_{0}^{-1}(y) \, \underset{y}{\nabla} G_{0}(x,y) \right) \, dy.
\end{equation*} 
By a standard argument, i.e. isolating the singularity, we deduce that $\large{\text{J}}_{0}(1) = -1/2$.
\bigskip
 
Let us now estimate the variation $\left( \large{\text{J}}_{\omega} - \large{\text{J}}_{0} \right)(1)(x)$.
We have
\begin{eqnarray*}
\left( \large{\text{J}}_{\omega} - \large{\text{J}}_{0} \right)(1)(x) &=& \int_{D} \underset{y}{\nabla} \cdot \left( \rho_{0}^{-1}(y) \underset{y}{\nabla} (G_{\omega} - G_{0})(x,y)\right) \, dy \\
&=& - \rho_{0}^{-1}(z) \, \kappa_{0}^{2} \, \int_{D} \Gamma_{0}(x,y) \, dy - \rho_{0}^{-1}(z) \, \kappa_{0}^{2} \, \int_{D} \left( \Gamma_{\omega} - \Gamma_{0} \right)(x,y) \, dy  \\
&+& \int_{D} \underset{y}{\nabla} \cdot \left( \rho_{0}^{-1}(y) \underset{y}{\nabla} G_{0} (x,y) - \rho_{0}^{-1}(z)  \underset{y}{\nabla}\Gamma_{0}(x,y)\right) \, dy  \\
&+&  \int_{D} \underset{y}{\nabla} \cdot \left( \rho_{0}^{-1}(y) \underset{y}{\nabla} G_{\omega} (x,y) - \rho_{0}^{-1}(z)  \underset{y}{\nabla}\Gamma_{\omega}(x,y)\right) \, dy. 
\end{eqnarray*} 
Since $\vert x - y \vert$ is small, clearly, we have
$\rho_{0}^{-1}(z) \, \kappa_{0}^{2} \, \int_{D} \left( \Gamma_{\omega} - \Gamma_{0} \right)(x,y) \, dy = \mathcal{O}\left( \varepsilon^{3} \right)$
and from $(\ref{ha})$, we deduce that $\int_{D} \underset{y}{\nabla} \cdot \left( \rho_{0}^{-1}(y) \underset{y}{\nabla} G_{\omega} (x,y) - \rho_{0}^{-1}(z)  \underset{y}{\nabla}\Gamma_{\omega}(x,y)\right) \,dy $ and $\int_{D} \underset{y}{\nabla} \cdot \left( \rho_{0}^{-1}(y) \underset{y}{\nabla} G_{0} (x,y) - \rho_{0}^{-1}(z)  \underset{y}{\nabla}\Gamma_{0}(x,y)\right) \, dy$ behave as  $\varepsilon^{3}$. Then 
\begin{equation*}
\left( \large{\text{J}}_{\omega} - \large{\text{J}}_{0} \right)(1)(x) = - \rho_{0}^{-1}(z) \, \kappa_{0}^{2} \, \int_{D} \Gamma_{0}(x,y) \, dy  + \mathcal{O}\left( \varepsilon^{3} \right)
\end{equation*}
moreover
\begin{equation*}
\rho_{0}^{-1}(z) \, \Gamma_{0}(x,y) = \frac{1}{4 \pi \vert x-y \vert} = \frac{-1}{2} \underset{y}{\nabla} \cdot \left( \frac{(x-y)}{4\pi \, \vert x-y \vert} \right)
\end{equation*} 
then a simple integration ends the proof of Lemma \ref{lemma1}.
\bigskip

Let us move to the proof of Lemma \ref{lemma-B}.
We investigate first the invertibility of $\left(\lambda + \frac{1}{2} \right) I+\large{\text{J}}_{0}$.\\ 
Let $f \in L^{2}(\partial D)$
such that $f \neq 0$ and
$\left(\lambda + \frac{1}{2} \, I + \large{\text{J}}_{0}^{\; *} \right)(f)=0$, then we have  
\begin{equation*}
0  =  \int_{\partial D} \left( (\lambda + \frac{1}{2}) \, I + \large{\text{J}}_{0}^{\; *} \right)(f) \, 1 \, d \sigma = \int_{\partial D} f \, \left( (\lambda + \frac{1}{2}) \, I + \large{\text{J}}_{0} \right)(1) \, d \sigma = \lambda \int_{\partial D} f \, d \sigma.
\end{equation*}
With standard argument, see for instance \cite{Habib-book}, we show that $\left(\lambda + \frac{1}{2} \right) I + \large{\text{J}}_{0}$ is invertible in $L^{2}(\partial D)$ with $\left\Vert \left( \left(\lambda + \frac{1}{2} \right) I + \large{\text{J}}_{0} \right)^{-1} \right\Vert_{\mathcal{L}\left(L^{2}(\partial D)\right)} = \mathcal{O}\left( 1/\lambda \right) $ and $ \left\Vert \left( \left(\lambda + \frac{1}{2} \right) I + \large{\text{J}}_{0} \right)^{-1} \right\Vert_{\mathcal{L}\left(L^{2}_{0}(\partial D)\right)} = \mathcal{O}\left( 1 \right), \; \, \text{uniformly on }   \lambda$.\\ 
Next, we investigate the invertibility of $\left(\lambda + \frac{1}{2} \right) I + \large{\text{J}}_{\omega}$ in $L^{2}(\partial D)$. For this, we need to compute $\left\Vert \large{\text{J}}_{\omega} - \large{\text{J}}_{0}\right\Vert_{\mathcal{L}}$.\\
We have 
\begin{eqnarray}\label{jwj0eps2}
\nonumber
\left\Vert \large{\text{J}}_{\omega} - \large{\text{J}}_{0}\right\Vert_{\mathcal{L}\left(L^{2}(\partial D)\right)} & = & \left\Vert \left( \large{\text{J}}_{\omega} - \large{\text{J}}_{0} \right)^{*}   \right\Vert_{\mathcal{L}\left(L^{2}(\partial D)\right)} := \underset{\Vert f \Vert = 1}{\sup} \left\Vert \left( \large{\text{J}}_{\omega} - \large{\text{J}}_{0} \right)^{*}\left( f \right)   \right\Vert = \underset{\Vert f \Vert = 1}{\sup} \left\Vert \rho_{0}^{-1} \, \left(K_{D}^{w} - K_{D}^{0} \right)^{*}(f) \right\Vert \\
& \lesssim & \underset{\Vert f \Vert = 1}{\sup} \left\Vert \left(K_{D}^{w} - K_{D}^{0} \right)^{*}(f) \right\Vert \stackrel{\ref{KwK0}}{\simeq} \underset{\Vert f \Vert = 1}{\sup} \left\Vert \int_{\partial D} \frac{\nu(\cdot) \cdot (\cdot - y)}{\vert \cdot - y \vert} \, f(y) \, d\sigma(y) \right\Vert = \mathcal{O}\left( \varepsilon^{2} \right).
\end{eqnarray}
Then 
\begin{eqnarray}\label{invneumann}
\nonumber
\left(\lambda + \frac{1}{2} \right) I + \large{\text{J}}_{\omega} &=& \left(\lambda + \frac{1}{2} \right) I + \large{\text{J}}_{0} -  \large{\text{J}}_{0} + \large{\text{J}}_{\omega} \\ \nonumber 
\left(\lambda + \frac{1}{2} \right) I + \large{\text{J}}_{\omega} & = & \left( \left(\lambda + \frac{1}{2} \right) I + \large{\text{J}}_{0} \right) \left[I - \left( \left(\lambda + \frac{1}{2} \right) I + \large{\text{J}}_{0} \right)^{-1} \left( \large{\text{J}}_{0} - \large{\text{J}}_{\omega} \right) \right] \\ 
\left( \left(\lambda + \frac{1}{2} \right) I + \large{\text{J}}_{\omega} \right)^{-1} 
& = & \left[I - \left( \left(\lambda + \frac{1}{2} \right) I + \large{\text{J}}_{0} \right)^{-1} \left( \large{\text{J}}_{0} - \large{\text{J}}_{\omega} \right) \right]^{-1} \left( \left(\lambda + \frac{1}{2} \right) I + \large{\text{J}}_{0} \right)^{-1},  
\end{eqnarray}
we know that $\left( \left(\lambda + \frac{1}{2} \right) I + \large{\text{J}}_{0} \right)^{-1}$ exists, then it suffices to prove that the first operator on the right hand side exists also. Using $(\ref{jwj0eps2})$ and assuming that $\mathbf{\bm{\lambda}^{-1} \, \bm{\varepsilon}^{2} <1}$, we have     
\begin{equation*}
\left\Vert \left( \left(\lambda + \frac{1}{2} \right) I + \large{\text{J}}_{0} \right)^{-1} \left( \large{\text{J}}_{0} - \large{\text{J}}_{\omega} \right) \right\Vert  \lesssim  \frac{1}{\dist\left(\lambda + \frac{1}{2} ; \sigma\left(\large{\text{J}}_{0}\right) \right)} \; \varepsilon^{2} \leq \lambda^{-1} \, \varepsilon^{2} < 1,
\end{equation*}
then by the \emph{Neumann} series representation for the inverse  operator we deduce that the first operator on the right hand side   of $(\ref{invneumann})$ exists and consequently $\left( \left(\lambda + \frac{1}{2} \right) I + \large{\text{J}}_{\omega} \right)^{-1}$ is well defined. Again, by $(\ref{invneumann})$ we deduce that  
\begin{equation}\label{lambda+1/2}
\left\Vert \left( \left(\lambda + \frac{1}{2} \right) I + \large{\text{J}}_{\omega} \right)^{-1} \right\Vert_{\mathcal{L}\left(L^{2}(\partial D)\right)} \leq \lambda^{-1}.  
\end{equation}
Similar arguments allow to obtain the following estimation as well 
\begin{equation*}
\left\Vert \left( \left(\lambda + \frac{1}{2} \right) I + \large{\text{J}}_{\omega} \right)^{-1} \right\Vert_{\mathcal{L}\left(L^{2}_{0}(\partial D)\right)} = \mathcal{O}\left( 1 \right), \; \, \text{uniformly on }   \lambda.  
\end{equation*}
From $\large{\text{J}}_{\omega}$ we construct the operator $B$ defined from $L^{2}(\partial D)$ to  $L^{2}(\partial D)$ as   
\begin{equation*}
\forall \, f \in L^{2}(\partial D), \; x \in \partial D,\;\;   B(f)(x) :=  \left[ \left( \rho_{0}^{-1}(x) \, \alpha^{-1}(z) + \frac{1}{2}  \right) I + \large{\text{J}}_{\omega}^{\; *} \right] (f)(x)
\end{equation*}
where we recall that $\alpha(z) := \rho_{1}^{-1} - \rho_{0}^{-1}(z)$ with $ \rho_{1} = 
\overline{\rho}_{1} \varepsilon^{2}$. 
\bigskip

We have  
\begin{eqnarray}\label{operatorB}
\nonumber
B &=&  \left[ \left( \rho_{0}^{-1}(z) \, \alpha^{-1}(z) + \frac{1}{2}  \right) I + \large{\text{J}}_{\omega}^{\; *} \right]  + \alpha^{-1}(z) \, \int_{0}^{1}(\cdot - z)\cdot \nabla \rho_{0}^{-1}(z+t(\cdot -z))dt\, I  \\ \nonumber
 &=&  \left[ \left( \rho_{0}^{-1}(z) \, \alpha^{-1}(z) + \frac{1}{2}  \right) I + \large{\text{J}}_{\omega}^{\; *} \right] \Bigg[ I \\ \nonumber && \qquad + \quad  \alpha^{-1}(z) \,\left[ \left( \rho_{0}^{-1}(z) \, \alpha^{-1}(z) + \frac{1}{2}  \right) I + \large{\text{J}}_{\omega}^{\; *} \right]^{-1} \, \int_{0}^{1}(\cdot -z)\cdot \nabla \rho_{0}^{-1}(z+t(\cdot -z))dt\, I  \Bigg]  \\ \nonumber
\mbox{ Then }\\
B^{-1} &=& \left[ I + \alpha^{-1}(z) \,\left[ \left( \rho_{0}^{-1}(z) \, \alpha^{-1}(z) + \frac{1}{2}  \right) I + \large{\text{J}}_{\omega}^{\; *} \right]^{-1} \, \int_{0}^{1}(\cdot -z)\cdot \nabla \rho_{0}^{-1}(z+t(\cdot -z))dt\, I  \right]^{-1}\\ \nonumber && \left[ \left( \rho_{0}^{-1}(z) \, \alpha^{-1}(z) + \frac{1}{2}  \right) I + \large{\text{J}}_{\omega}^{\; *} \right]^{-1}.  
\end{eqnarray}
We know that $\left[ \left( \rho_{0}^{-1}(z) \, \alpha^{-1}(z) + \frac{1}{2}  \right) I + \large{\text{J}}_{\omega}^{\; *} \right]^{-1}$ exists if $\left(\rho_{0}^{-1}(z) \, \alpha^{-1}(z)\right)^{-1} \varepsilon^{2} < 1$ or equivalently if $\rho_{0}(z) < (1 + \varepsilon^{2}) \, \overline{\rho}_{1}$, but recall that we have assumed $\overline{\rho}_{1}$ large enough such that the previous condition is satisfied. Next, it is sufficient to prove that     
\begin{equation*}
\zeta := \left\Vert \alpha^{-1}(z) \,\left[ \left( \rho_{0}^{-1}(z) \, \alpha^{-1}(z) + \frac{1}{2}  \right) I + \large{\text{J}}_{\omega}^{\; *} \right]^{-1} \, \int_{0}^{1}(\cdot -z)\cdot \nabla \rho_{0}^{-1}(z+t(\cdot -z))dt\, I \right\Vert_{L^{2}(\partial D)} 
\end{equation*}
is less than $1$. For this, thanks to $(\ref{lambda+1/2})$, we can prove that $\zeta =\mathcal{O}\left( \varepsilon \right)$. Finally, $B$ is invertible.  \\ 
From $(\ref{operatorB})$, we have 
\begin{equation*}
\Vert B^{-1} \Vert_{L^{2}(\partial D)} \lesssim \frac{1}{1 - \zeta }  \; \left\Vert \left[ \left( \rho_{0}^{-1}(z) \, \alpha^{-1}(z) + \frac{1}{2}  \right) I + \large{\text{J}}_{\omega}^{\; *} \right]^{-1} \right\Vert_{L^{2}(\partial D)} \simeq \, \alpha(z).
\end{equation*}
With the same arguments we can prove that $\Vert B^{-1} \Vert_{L^{2}_{0}(\partial D)} = \mathcal{O}(1)$.




\end{document}